\documentclass[notitlepage,11pt,reqno]{amsart}
\usepackage[foot]{amsaddr}
\usepackage[numbers,sort]{natbib}

\usepackage{amssymb,nicefrac,bm,upgreek,mathtools,verbatim}
\usepackage[final]{hyperref}
\usepackage[mathscr]{eucal}
\usepackage{dsfont}

\usepackage{color}
\definecolor{dmagenta}{rgb}{.4,.1,.5}
\definecolor{dblue}{rgb}{.0,.0,.5}
\definecolor{mblue}{rgb}{.0,.0,.8}
\definecolor{ddblue}{rgb}{.0,.0,.4}
\definecolor{dred}{rgb}{.6,.0,.0}
\definecolor{dgreen}{rgb}{.0,.5,.0}
\definecolor{Eeom}{rgb}{.0,.0,.5}

\usepackage[normalem]{ulem}

\usepackage[margin=1in]{geometry}
\allowdisplaybreaks

\newtheorem{lemma}{Lemma}[section]
\newtheorem{theorem}{Theorem}[section]

\newtheorem{corollary}{Corollary}[section]

\theoremstyle{definition}

\newtheorem{assumption}{Assumption}[section]

\newtheorem{example}{Example}[section]
\newtheorem{remark}{Remark}[section]

\numberwithin{equation}{section}

\hypersetup{
  colorlinks=true,
  citecolor=dblue,
  linkcolor=dblue,
  frenchlinks=false,
  pdfborder={0 0 0},
  naturalnames=false,
  hypertexnames=false,
  breaklinks}
\usepackage[capitalize,nameinlink]{cleveref}

\crefname{section}{Section}{Sections}
\crefname{subsection}{Subsection}{Subsections}
\crefname{condition}{Condition}{Conditions}
\crefname{hypothesis}{Hypothesis}{Conditions}
\crefname{assumption}{Assumption}{Assumptions}
\crefname{lemma}{Lemma}{Lemmas}
\crefname{claim}{Claim}{Claims}

\Crefname{figure}{Figure}{Figures}

\crefformat{equation}{\textup{#2(#1)#3}}
\crefrangeformat{equation}{\textup{#3(#1)#4--#5(#2)#6}}
\crefmultiformat{equation}{\textup{#2(#1)#3}}{ and \textup{#2(#1)#3}}
{, \textup{#2(#1)#3}}{, and \textup{#2(#1)#3}}
\crefrangemultiformat{equation}{\textup{#3(#1)#4--#5(#2)#6}}%
{ and \textup{#3(#1)#4--#5(#2)#6}}{, \textup{#3(#1)#4--#5(#2)#6}}%
{, and \textup{#3(#1)#4--#5(#2)#6}}

\Crefformat{equation}{#2Equation~\textup{(#1)}#3}
\Crefrangeformat{equation}{Equations~\textup{#3(#1)#4--#5(#2)#6}}
\Crefmultiformat{equation}{Equations~\textup{#2(#1)#3}}{ and \textup{#2(#1)#3}}
{, \textup{#2(#1)#3}}{, and \textup{#2(#1)#3}}
\Crefrangemultiformat{equation}{Equations~\textup{#3(#1)#4--#5(#2)#6}}%
{ and \textup{#3(#1)#4--#5(#2)#6}}{, \textup{#3(#1)#4--#5(#2)#6}}%
{, and \textup{#3(#1)#4--#5(#2)#6}}

\crefdefaultlabelformat{#2\textup{#1}#3}

\makeatletter
\DeclareRobustCommand\widecheck[1]{{\mathpalette\@widecheck{#1}}}
\def\@widecheck#1#2{%
    \setbox\z@\hbox{\m@th$#1#2$}%
    \setbox\tw@\hbox{\m@th$#1%
       \widehat{%
          \vrule\@width\z@\@height\ht\z@
          \vrule\@height\z@\@width\wd\z@}$}%
    \dp\tw@-\ht\z@
    \@tempdima\ht\z@ \advance\@tempdima2\ht\tw@ \divide\@tempdima\thr@@
    \setbox\tw@\hbox{%
       \raise\@tempdima\hbox{\scalebox{1}[-1]{\lower\@tempdima\box
\tw@}}}%
    {\ooalign{\box\tw@ \cr \box\z@}}}

\def\subsection{\@startsection{subsection}{0}%
\z@{\linespacing\@plus\linespacing}{\linespacing}%
{\bf}}

\makeatother

\DeclareMathOperator{\Exp}{\mathbb{E}} 
\DeclareMathOperator{\Prob}{\mathbb{P}} 
\newcommand{\D}{\mathrm{d}}          
\newcommand{\RR}{\mathbb{R}}         
\newcommand{\Rd}{{\mathbb{R}^d}}       
\newcommand{\Ind}{\mathds{1}}            

\newcommand{\sB}{\mathscr{B}}    
\newcommand{\sU}{\mathscr{U}}
\newcommand{\cC}{\mathcal{C}}     
\newcommand{\cD}{\mathcal{D}}     
\newcommand{\cF}{\mathcal{F}}
\newcommand{\cH}{\mathcal{H}}

\newcommand{\cK}{\mathcal{K}}
\newcommand{\sK}{\mathscr{K}}     
\newcommand{\sM}{\mathscr{M}}     
\newcommand{\cN}{\mathcal{N}}

\newcommand{\abs}[1]{\lvert#1\rvert}
\newcommand{\norm}[1]{\lVert#1\rVert}

\providecommand{\pro}[1]{(#1_t)_{t \geq 0}}

\providecommand{\semi}[1]{\{#1_t: t \geq 0\}}

\DeclareMathOperator{\Dom}{Dom}
\DeclareMathOperator{\Spec}{Spec}
\DeclareMathOperator{\diam}{diam}
\DeclareMathOperator{\inrad}{inrad}

\newcommand{\ex}{\mathbb{E}}

\newcommand{\Dst}{\mathrel{\stackrel{\makebox[0pt]{\mbox{\normalfont\tiny d}}}{=}}}

\DeclareMathOperator*{\supp}{supp}
\DeclareMathOperator*{\Int}{Int}

\DeclareMathOperator{\dist}{dist}

\DeclareMathOperator{\Psidel}{\Psi(-\Delta)}

\begin{document}

\title[Universal Constraints]%
{\sc \textbf{Universal Constraints on the Location of Extrema of Eigenfunctions of Non-Local Schr\"odinger Operators}}

\author{Anup Biswas and J\'ozsef L\H{o}rinczi}

\address{Anup Biswas \\
Department of Mathematics, Indian Institute of Science Education and Research, Dr. Homi Bhabha Road,
Pune 411008, India, anup@iiserpune.ac.in}

\address{J\'ozsef L\H{o}rinczi \\
Department of Mathematical Sciences, Loughborough University, Loughborough LE11 3TU, United Kingdom,
J.Lorinczi@lboro.ac.uk}

\date{}


\begin{abstract}
We derive a lower bound on the location of global extrema of eigenfunctions for a large class of non-local
Schr\"odinger operators in convex domains under Dirichlet exterior conditions, featuring the symbol of the
kinetic term, the strength of the potential, and the corresponding eigenvalue, and involving a new universal
constant. We show a number of probabilistic and spectral geometric implications, and derive a Faber-Krahn
type inequality for non-local operators. Our study also extends to potentials with compact support, and we
establish bounds on the location of extrema relative to the boundary edge of the support or level sets around
minima of the potential.
\end{abstract}



\maketitle

\section{\textbf{Introduction}}
Recently, in the paper \cite{RS-17} the remarkable bound
\begin{equation}
\label{RS}
\dist(x^*, \partial \cD) \geq c \, \Big\|\frac{\Delta \varphi}{\varphi} \Big\|^{-1/2}_{L^\infty(\cD)}
\end{equation}
has been obtained on the distance between the location of an assumed global maximum $x^*$ of any eigenfunction $\varphi$
of the Schr\"odinger operator $H = - \Delta + V$ with Dirichlet boundary condition set for a simply connected domain
$\cD \subset \RR^2$, and the boundary of $\cD$. Here the potential $V$ is bounded and includes the eigenvalue corresponding
to $\varphi$, and $c > 0$ is a constant, independent of $\cD$, $\varphi$ and $V$. The paper \cite{Biswas-17} established
a similar relationship for the fractional Schr\"odinger operator $(-\Delta)^{\alpha/2} + V$, $0 < \alpha < 2$, for arbitrary
dimensions $d \geq 2$, and pointed out some interesting corollaries.

In this paper we consider the problem of the location of extrema in a substantially amplified context and set of goals.
While we make use of an inspiring basic idea leading to \eqref{RS}, we see it worthwhile to be developed to a far greater
extent than attempted by the authors in \cite{RS-17}, in order to serve as the beginning of a programme of studying important
aspects of local behaviour for a whole class of equations of pure and applied interest. Specifically, our framework is the
class of non-local Schr\"odinger operators of the form
\begin{equation}
\label{nonlocSch}
H = \Psidel + V,
\end{equation}
where $\Psi$ is a so-called Bernstein function, and $V$ is a multiplication operator called potential (for details see Section 
2). Such operators have been considered from a combined perturbation theory and functional integration point of view in 
\cite{HIL12,HL-12,KKM16}, and we will discuss some motivations below. When $\Psi$ is the identity function, we get back to 
classical Schr\"odinger operators, thus this framework also allows comparison with other operators and related equations, and 
new light is shed also on classical Laplacians with or without potentials.

We will be interested in the properties of solutions of two eigenvalue problems. The first is a non-local Dirichlet-Schr\"odinger
problem for a bounded convex domain $\cD \subset \Rd$, $d \geq 1$, given by
\begin{align}
\label{dirichev}
\left\{
\begin{array}{cc}
H \varphi =  \lambda \varphi & \mbox{in $\cD$}, \\
\varphi = 0 & \, \mbox{in $\cD^c$},
\end{array}
\right.
\end{align}
in weak sense. In this case there is a countable set of eigenvalues
$$
\lambda_1^{\cD,V} < \lambda_2^{\cD,V} \leq \lambda_3^{\cD,V} \leq ...
$$
of finite multiplicities each, and a corresponding orthonormal set of eigenfunctions $\varphi_1, \varphi_2, ...
\in \Dom(H) \subset L^2(\cD)$. When $V \equiv 0$, the problem reduces to the non-local Dirichlet eigenvalue equation.
In this case the spectrum is still discrete, and we use the notation $\lambda_1^\cD, \lambda_2^\cD, ...$ for the
eigenvalues.

The second problem we consider is the eigenvalue equation in $L^2(\Rd)$, $d \geq 1$,
\begin{equation}
\label{fullev}
H\varphi = \lambda \varphi, \quad \supp V = \cK, \; \mbox{with $\cK \subset \Rd$ bounded.}
\end{equation}
In particular, this covers potential wells of depth $v > 0$, when $V = - v \Ind_\cK$, which are of basic interest.
In this case appropriate conditions will be needed on $V$ in order to have any $L^2$-eigenfunctions. The
Dirichlet-Schr\"odinger problem can also be seen as a Schr\"odinger problem in full space, where the potential equals
$V$ in $\cD$ and infinity elsewhere.

Non-local (i.e., integro-differential) equations are gaining increasing interest recently from the corners of both pure and
applied mathematics. While PDE based on the Laplacian and related elliptic operators proved to be ubiquitous in virtually
every fundamental model of dynamics for a long time, it is now recognized that a new range of effects is captured if one uses
a class of non-local operators, in which the classical Laplacian is just one special case. Much work has been done recently on
the well-posedness and regularity theory of such equations, see, e.g., \cite{CL09,ROS14} and many related references. There is
also much interest due to the fact that non-local operators are generators of L\'evy or Feller processes \cite{J,BSW}, and their
study is made possible by probabilistic and potential theory methods. On the other hand, applications to mathematical physics,
such as anomalous transport \cite{KRS08,MS11} and quantum theory \cite{D83,LS10,HIL13}, or more computationally, image reconstruction
via denoising \cite{GO09,BCM10}, to name just a few, provide a continuing incentive to the development of these ideas and techniques.

The study of non-local Schr\"odinger equations is one aspect of this work, and some primary work on developing a related
potential theory has been done in \cite{BB99,BB00}. There are many possible choices of $\Psi$ of interest in applications.
The fractional Laplacian $\Psi(-\Delta) = (-\Delta)^{\alpha/2}$, $0 < \alpha < 2$, is the most studied of them. There is
a range of exponents $\alpha$ used in anomalous transport theory, however, there are many further applications,
e.g., $\alpha = 1.3$ describes the dynamics of particles trapped in the vortices of a flow, $\alpha = 1.5$ relates with
the spatial distribution of the gravitational field generated by a cluster of uniformly distributed stars, etc. The
relativistic Laplacians $\Psi(-\Delta) = (-\Delta + m^{2/\alpha})^{\alpha/2} - m$, $m > 0$, are used to describe relativistic
or photonic quantum effects, geometric stable operators $\Psi(-\Delta) = \log(1+(-\Delta)^{\alpha/2})$, $0 < \alpha \leq 2$,
are more used in studying financial processes \cite{RM}, and so on. For a more detailed discussion we refer to \cite{KL16}.
We also note that qualitatively different spectral and analytic behaviours of $H$ and the related semigroup occur in function
of the choice of $\Psi$. For instance, for operators and related processes for which the singular integral (L\'evy jump)
kernel is polynomially or sub-exponentially decaying (e.g., fractional Laplacian), the eigenfunctions have very different
asymptotic behaviours than for exponentially or super-exponentially decaying kernels (e.g., relativistic Laplacian), and a
phase transition-like phenomenon occurs (for details see \cite[Sect. 4.4]{KL17}). All this shows that these operators, in
diverse aspects, massively differ from the classical Laplacian, and also produce spectacular differences between themselves.

Explicit solutions of eigenvalue problems for non-local Schr\"odinger operators are rare. In \cite{LM12} this has been obtained for
the operator $(-d^2/dx^2)^{1/2} + x^2$, and in \cite{DL16} for $(-d^2/dx^2)^{1/2} + x^4$, both in $L^2(\RR)$. A detailed study of
the asymptotic behaviour at infinity of the eigenfunctions for a large class of non-local Schr\"odinger operators has been
made in \cite{KL15,KL17}. Bounds, monotonicity and continuity properties for Dirichlet eigenvalues for large classes of domains
have been established in \cite{Chen-Song,CS06}, approximate solutions and detailed estimates for some non-local Dirichlet problems
in intervals, half-spaces or boxes have been presented in \cite{KKMS10,K11,K12,KKM13,KKM16}. Since an explicit computation of the
principal Dirichlet eigenvalue or eigenfunction is not available even for the simplest cases, a study of the properties of
the spectrum becomes important. Some results on the shape of eigenfunctions or solutions were obtained in \cite{BKM06,K17}, and
\cite{LR17} investigates the local behaviour of eigenfunctions for potentials wells.

Our results in this paper contribute to a study of local properties of eigenfunctions of non-local Schr\"odinger operators.
There are several reasons why information on the location of extrema of eigenfunctions is of interest, and we single out a
few here as follows:
\begin{enumerate}
\item[(i)]
\emph{Maximum principles:} Maximum principles are fundamental tools in the study of elliptic and parabolic linear, and
semilinear problems. Using our present work and the techniques developed here for non-local operators, we are able to derive
and prove elliptic and parabolic Aleksandrov-Bakelman-Pucci type estimates, Berestycki-Nirenberg-Varadhan type refined
maximum principles, anti-maximum principles in the sense of Cl\'ement-Peletier, a maximum principle for narrow domains, as
well as Liouville theorems. This is presented in detail elsewhere, see \cite{BL2}, and for the context of time-fractional
evolutions \cite{BL3}. We note that using our techniques all this could be implemented in the framework of viscosity solutions.

\item[(ii)]
\emph{Hot-spots:}
A hot-spot is a point in space where the solution of the heat equation in a bounded domain at a given time attains its maximum, and
an object of study for classical domain Laplacians has been how they move in time when Neumann or Dirichlet boundary conditions are
imposed. For Dirichlet boundary conditions, on the long run the solution increasingly takes the shape of the principal eigenfunction,
and the hot-spot becomes its maximizer. While there are several classical results on this challenging problem, we mention \cite{GJ-98},
in which the problem is studied for bounded convex sets in $\RR^2$, and the recent paper \cite{BMS-11} which obtained a lower bound
on the location of the maximum of the principal Dirichlet eigenfunction. One implication of our results is a significantly
improved bound, see a discussion in Remark \ref{hotspots} and Corollary \ref{HScorr} below.

\item[(iii)]
\emph{Torsion:} The torsion function is the solution of a specific Dirichlet boundary value problem, with interest originally derived
from mechanics and also having an important probabilistic meaning. A puzzling phenomenon is that its maximizer and the maximizer of
the principal Dirichlet eigenfunction of the Laplacian are located very near to each other, though they fail to coincide, see
\cite{BLMS,HLP} and the references therein. In our present work we also obtain a result on this for the non-local case, for a further
discussion see Remark \ref{tors}.

\item[(iv)]
\emph{Most likely location of paths:}
Since the principal eigenfunction of $H$ can be chosen to be strictly positive, by its harmonicity it can be used as a Doob $h$-transform
to construct a stochastic process obtained under the perturbation of $V$ of the subordinate Brownian motion generated by $\Psi(-\Delta)$.
In case of a classical Schr\"odinger operator this is a diffusion, while for non-local cases it is a L\'evy-type jump process with, in
general, unbounded coefficients. In both cases the maximizer of the first eigenfunction gives the mode of the stationary probability density
of this process, i.e., describes the location in space which gives the highest contribution into the distribution of paths. This is discussed
in further detail in Remark \ref{doob} below.

\item[(v)]
\emph{An application to modelling groundwater contamination:}
An application of high practical interest of anomalous transport described by non-local equations is a more realistic description of the
spread of contaminated groundwater by taking into account non-uniformities of a porous soil, see \cite{KM} and references therein. The
maximizer(s) in this case have a relevance in the localization of the highest-concentration points of the plume. Also, in this context
the study of inverse problems become important in order to control the level sets and maxima of the plume. Using our techniques we have
been able to discuss an inverse source problem in \cite[Th. 3.6]{BL3}, which is just the beginning of a series of investigations of a
practical relevance.
\end{enumerate}

To conclude, we outline the main results and highlight some technical achievements in this paper, apart from what we discussed above.
\begin{enumerate}
\item
Our key results are stated in Theorems \ref{T3.20}-\ref{T3.1}, and we will study their multiple implications involving
an interplay of probabilistic and spectral geometric aspects. Our results reproduce \eqref{RS} as a specific case, however,
apart from a far more general framework our work here goes well beyond \cite{RS-17} on several counts. One is that our expressions 
feature the symbol of the kinetic part of the operator. This allows us to understand what is behind the formulae involving a 
\emph{lower bound} on the position of extrema, and it will turn out from the probabilistic representation that it results 
from a balance of two survival times of paths of the related random process running in the domain (Remark \ref{doob}). This 
points to an underlying mechanism fundamentally involving the competition of energy versus entropy effects, and offers a 
very different perspective. A second point how we reach a different level of discussion is that we allow a large class of
potentials, including local singularities, and do not limit ourselves to bounded potentials. This has not been attempted in
\cite{RS-17}, and controlling such a possibly very ``rugged" potential landscape is not a straightforward step from bounded 
potentials. We also note that our combined analytic and probabilistic techniques developed here allow to cover general convex 
domains (see Theorem \ref{FKsemi} and Remark \ref{R3.1}), removing boundary regularity problems often encountered when using 
purely analytic means. This will also be helpful when considering maximum principles in \cite{BL2}.

\item
Apart from bounded domains we also consider potentials with compact support in full space $\Rd$, and derive predictions
on the location of extrema relative to the edge of their supports or from neighbourhoods (e.g., level sets) of the minima
of the potential, see Section 4 and specifically the key Theorems \ref{T3.2} and \ref{T3.4} below. There is very little 
information on this even for the classical and, as far as we are aware, nothing for non-local Schr\"odinger operators. We 
also note that this problem has not been addressed in \cite{RS-17}. As consequences, we observe some interesting behaviours 
dependent on whether the potential is attracting or repelling (Theorem \ref{T3.3}).

\item
As it will be seen in what follows, the localization of the extrema of Dirichlet-Schr\"odinger eigenfunctions is, roughly
speaking, an isoperimetric-type property, determined by underlying geometric principles. In Corollary \ref{FKineq} we
obtain a new Faber-Krahn type inequality for non-local Schr\"odinger operators as a direct consequence of the estimates on
the location of extrema.

\item
We also obtain a variety of geometric and probabilistic bounds on the eigenvalues in the spirit of the discussions in
\cite{BC-94,BLM-01} and references therein. In particular, we derive a lower estimate on all moments of exit times of
subordinate Brownian motion from convex domains, and further relations on eigenvalues (Corollaries \ref{C3.1} and
\ref{C3.2}).
\end{enumerate}

The remainder of this paper is organized as follows. In Section 2 we discuss some properties of Bernstein functions $\Psi$ on
which we rely throughout below when using the operators $\Psi(-\Delta)$ and related subordinate Brownian motions. In Section 3
first we establish some basic facts on the Dirichlet-Schr\"odinger eigenvalue problem, which do not seem to be available in the
literature. Next we state and prove our main results in Theorems \ref{T3.20}-\ref{T3.1}, and then discuss a number of consequences
and implications in corollaries and a string of remarks. Section 4 is devoted to operators having potentials with compact support.

\section{\textbf{Bernstein functions of the Laplacian and subordinate Brownian motions}}
Now we turn to describe the above objects formally. Denote
\begin{equation}
\label{BernLapl}
H_0 = \Psidel,
\end{equation}
where $\Psi$ is a Bernstein function given below. This operator can be defined via functional calculus by using
the spectral decomposition of the Laplacian. It is a pseudo-differential operator whose symbol is given by the
Fourier multiplier
$$
\widehat{H_0 f}(y) = \Psi(|y|^2)\widehat f(y), \quad y \in \Rd, \; f \in  \Dom(H_0),
$$
with domain $\Dom(H_0)=\big\{f \in L^2(\Rd): \Psi(|\cdot|^2) \widehat f \in L^2(\Rd) \big\}$. It follows by
general arguments that $H_0$ is a positive, self-adjoint operator with core $C_{\rm c}^\infty(\Rd)$, for
details see \cite{HIL12, SSV}.

Recall that a Bernstein function is a non-negative completely monotone function, i.e., an element of the set
$$
\mathcal B = \left\{f \in C^\infty((0,\infty)): \, f \geq 0 \;\; \mbox{and} \:\; (-1)^n\frac{d^n f}{dx^n} \leq 0,
\; \mbox{for all $n \in \mathbb N$}\right\}.
$$
In particular, Bernstein functions are increasing and concave. We will make use below of the subset
$$
{\mathcal B}_0 = \left\{f \in \mathcal B: \, \lim_{u\downarrow 0} f(u) = 0 \right\}.
$$
Let $\mathcal M$ be the set of Borel measures $\mu$ on $\RR \setminus \{0\}$ with the property that
$$
\mu((-\infty,0)) = 0 \quad \mbox{and} \quad \int_{\RR\setminus\{0\}} (y \wedge 1) \mu(dy) < \infty.
$$
Notice that, in particular, $\int_{\RR\setminus\{0\}} (y^2 \wedge 1) \mu(dy) < \infty$ holds, thus $\mu$ is a L\'evy
measure supported on the positive semi-axis. It is well-known then that every Bernstein function $\Psi \in
{\mathcal B}_0$ can be represented in the form
\begin{equation}
\Psi(u) = bu + \int_{(0,\infty)} (1 - e^{-yu}) \mu(\D{y})
\end{equation}
with $b \geq 0$, moreover, the map $[0,\infty) \times \mathcal M \ni (b,\mu) \mapsto \Psi \in {\mathcal B}_0$ is
bijective. $\Psi$ is said to be a complete Bernstein function if there exists a Bernstein function $\widetilde\Psi$
such that
$$
\Psi(u)= u^2 \mathcal{L}(\widetilde\Psi)(u), \quad u>0\,,
$$
where $\mathcal{L}$ stands for Laplace transform. It is known that every complete Bernstein function is also
a Bernstein function. Also, for a complete Bernstein function the L\'evy measure $\mu(\D{y})$ has a completely monotone
density with respect to the Lebesgue measure. The class of complete Bernstein functions is large, including important
cases such as
\begin{itemize}
\item[(i)]
$\Psi(u) = u^{\alpha/2}$, $\alpha\in(0, 2]$
\item[(ii)]
$\Psi(u) = (u+m^{2/\alpha})^{\alpha/2}-m$, $m\geq 0$, $\alpha \in (0, 2)$
\item[(iii)]
$\Psi(u) = u^{\alpha/2} + u^{\beta/2}$,  $0 < \beta < \alpha \in(0, 2]$
\item[(iv)]
$\Psi(u) = \log(1+u^{\alpha/2})$, $\alpha\in (0,2]$
\item[(v)]
$\Psi(u) = u^{\alpha/2}(\log(1+u))^{\beta/2}$, $\alpha \in (0,2)$, $\beta \in (0, 2-\alpha)$
\item[(vi)]
$\Psi(u) = u^{\alpha/2} (\log(1+u))^{-\beta/2}$, $\alpha \in (0,2]$, $\beta \in [0,\alpha)$.
\end{itemize}
On the other hand, the Bernstein function $\Psi(u) = 1 - e^{-u}$ is not a complete Bernstein function. For a detailed
discussion we refer to the monograph \cite{SSV}.

Bernstein functions are closely related to subordinators, and we will use this relationship below. Recall that a
one-dimensional L\'evy process $\pro S$ on a probability space $(\Omega_S, {\mathcal F}_S, \mathbb P_S)$ is called a
subordinator whenever it satisfies $S_s \leq S_t$ for $s \leq t$, $\mathbb P_S$-almost surely.
A basic fact is that the Laplace transform of a subordinator is given by a Bernstein function, i.e.,
\begin{equation}
\label{lapla}
\mathcal{L}(S_t)(u)=\ex_{\mathbb P_S} [e^{-uS_t}] = e^{-t\Psi(u)}, \quad t \geq 0,
\end{equation}
holds, where $\Psi \in {\mathcal B}_0$. In particular, there is a bijection between the set of subordinators on a given
probability space and Bernstein functions with vanishing right limits at zero; to emphasize this, we will occasionally
write $\pro {S^\Psi}$ for the unique subordinator associated with Bernstein function $\Psi$. Corresponding to the
examples above, the related processes are (i) $\alpha/2$-stable subordinator, (ii) relativistic
$\alpha/2$-stable subordinator, (iii) sums of independent subordinators of different indeces, (iv) geometric $\alpha/2$-stable
subordinators (specifically, the Gamma-subordinator for $\alpha = 2$), etc. The non-complete Bernstein function mentioned
above describes the Poisson subordinator.

Let $\pro B$ be $\Rd$-valued a Brownian motion on Wiener space $(\Omega_W,{\mathcal F}_W, \mathbb P_W)$, running twice
as fast as standard $d$-dimensional Brownian motion, and let $\pro {S^\Psi}$ be an independent subordinator. The random
process
$$
\Omega_W \times \Omega_S \ni (\omega,\varpi) \mapsto B_{S_t(\varpi)}(\omega) \in \Rd
$$
is called subordinate Brownian motion under $\pro {S^\Psi}$. For simplicity, we will denote a subordinate Brownian motion
by $\pro X$, its probability measure for the process starting at $x \in \Rd$ by $\mathbb P^x$, and expectation with respect
to this measure by $\ex^x$. Every subordinate Brownian motion is a L\'evy process, with infinitesimal generator $H_0
= \Psidel$. Subordination then gives the expression
\begin{equation}
\label{subord}
\Prob (X_t \in E) = \int_0^\infty \Prob_W(B_s \in E)\Prob_S(S_t \in \D s),
\end{equation}
for every measurable set $E$.

Our main concern in what follows are some properties in the bulk of functions satisfying the eigenvalue equations
(\ref{dirichev}-\ref{fullev}) in weak sense. Specifically, we will focus on the location of extrema of eigenfunctions
by using a stochastic representation of the solutions, featuring subordinate Brownian motion.

\section{\textbf{Constraints on the location of extrema}}
\subsection{The Dirichlet-Schr\"odinger problem}

In this section we assume $\cD \subset \Rd$ to be a bounded open set. Consider a complete Bernstein function $\Psi$ and
the operator $H_0 = \Psi(-\Delta)$ on $L^2(\Rd)$. The Dirichlet eigenvalue problem \eqref{dirichev} for $V \equiv 0$ has
been studied in various papers, including \cite{Chen-Song,CS06,K11,K12,KKM16}. In particular, the following holds; for
details we refer to \cite{KKM16} and \cite{FOT}.
Consider the space $C_{\rm c}^\infty(\cD)$, and define the operator $H_0^\cD$ given by the Friedrichs extension of
$H_0|_{C_{\rm c}^\infty(\cD)}$. It can be shown that the form-domain of $H_0^\cD$ contains those functions that are in
the form-domain of $H_0$ and are almost surely zero outside of $\cD$. Furthermore, the operator $-H_0^\cD$ generates the
strongly continuous operator semigroup
$$
T^\cD_t = e^{-tH_0^\cD}, \quad t \geq 0.
$$
Each operator $T^\cD_t$ is a contraction on $L^p(\cD)$, for every $p \geq 1$, including $p=\infty$. When $\Psi$ is unbounded,
$T^\cD_t$ is a contraction also on $C_0(\cD)$. If $e^{-t\Psi(|x|^2)} \in L^1(\cD)$ for $t > 0$, then each $T_t^\cD$ is a
Hilbert-Schmidt operator, in particular, they are compact. Hence, by general theory, the equation
$$
T^\cD_t \varphi = e^{-\lambda t} \varphi, \quad t > 0,
$$
is solved by a countable set of eigenvalues $\lambda_1^\cD < \lambda_2^\cD \leq \lambda_3^\cD \leq \cdots \to \infty$,
of finite multiplicity each, corresponding to an orthonormal set of eigenfunctions $\varphi_1^\cD, \varphi_2^\cD, ...
\in L^2(\cD)$. The principal eigenvalue $\lambda_1^\cD$ has multiplicity one, and the principal eigenfunction
$\varphi_1^\cD$ has a strictly positive version, which we will adopt throughout.
Moreover, due to strong continuity of
the semigroup, the spectrum is independent of $t > 0$, in particular, since $-H_0^\cD$ is the infinitesimal generator
of $\semi {T^\cD}$, the same eigenvalues and eigenfunctions also solve \eqref{dirichev} for $V \equiv 0$. It is also
known that $\semi {T^\cD}$ is the Markov semigroup of killed subordinate Brownian motion, i.e., we have
\begin{equation}
\label{killed}
T^\cD_t f(x) = \ex^x[f(X_t)\Ind_{\{\uptau_D > t\}}], \quad x \in \cD, \, t > 0, \, f\in L^2(\cD),
\end{equation}
where
\begin{equation}
\label{exit}
\uptau_\cD = \inf\{t > 0: \, X_t \not\in \cD\}
\end{equation}
is the first exit time of $\pro X$ from $\cD$.

In contrast to the pure Dirichlet problem, the Dirichlet-Schr\"odinger problem \eqref{dirichev} with $V \not\equiv 0$
has been much less studied and the counterparts of the above facts do not seem to be readily available in the literature.
Let $V \in L^\infty(\Rd)$ and consider $H = \Psi(-\Delta) + V$. This operator is bounded from below, and self-adjoint on
the dense domain $\Dom(\Psi(-\Delta)) \subset L^2(\Rd)$, with core $C_{\rm c}^\infty(\Rd)$. For a bounded open set $\cD
\subset \Rd$ we define the non-local Schr\"odinger operator $H^{\cD,V}$ as the Friedrichs extension of
$H|_{C_{\rm c}^\infty(\cD)}$. Also, define
\begin{equation}
\label{killedV}
T^{\cD,V}_t f(x) = \ex^x[e^{-\int_0^t V(X_s)ds}f(X_t)\Ind_{\{\uptau_D > t\}}], \quad x \in \cD, \, t > 0, \, f\in L^2(\cD).
\end{equation}
We denote $L^p$ norm on $\cD$ by $\norm{\cdot}_{p, \cD}$, whereas $\norm{\cdot}_p$ denotes the $L^p$ norm on $\Rd$. We show
the following properties.
\begin{lemma}
\label{FKsemi}
Consider the operators $H^{\cD,V}$ and $T_t^{\cD,V}$, $t > 0$, and let $\pro {S^\Psi}$ be the subordinator corresponding to
the Bernstein function $\Psi \in \mathcal B_0$. Suppose that $\Psi$ satisfies the Hartman-Wintner condition
\begin{equation}
\label{HW}
\lim_{|u|\to \infty}\frac{\Psi(\abs{u}^2)}{\log\abs{u}}=\infty.
\end{equation}
The following hold:
\begin{enumerate}
\item[(i)]
Every $T_t^{\cD,V}$ is an integral operator and we have the representation
\begin{eqnarray}\label{EL4.3A}
T^{\cD,V}_t f(x)
&=&
\int_{\cD}\ex^0_{\Prob_S}
\left[p_{S^\Psi_t}(x-y) \ex^{x,y}_{0,S^\Psi_t}[e^{-\int_0^t V(B_{S^\Psi_s})ds}]\Ind_{\{\uptau_D > t\}}\right]f(y)dy \\
&=&
\int_{\cD} T^{\cD,V}(t,x,y) f(y)dy, \quad x \in \cD, \, t \geq 0, \, f\in L^2(\cD), \nonumber
\end{eqnarray}
where $p_t(x) = (4\pi t)^{-d/2}e^{-\frac{|x|^2}{4t}}$, and $\ex^{x,y}_{0,S^\Psi_t}$ denotes expectation with respect
to the Brownian bridge measure from $x$ at time 0 to $y$ at time $s$, evaluated at random time $s=S^\Psi_t$. Furthermore,
for every $t>0$, we have $T^{\cD, V}(t, x, y)= T^{\cD, V}(t, y, x)$ for all $x, y\in\Rd$.
\item[(ii)]
$\semi {T^{\cD,V}}$ is a strongly continuous semigroup on $L^p(\cD)$, $p \geq 1$, with infinitesimal generator $-H^{\cD,V}$.
\item[(iii)]
Every $T_t^{\cD,V}$ is a Hilbert-Schmidt operator on $L^2(\cD)$, for all $t > 0$.
\item[(iv)]
The map $(0, \infty) \times \cD \times \cD \ni (t,x,y) \mapsto T^{\cD,V}(t,x,y) \in \RR$ is continuous.
\item[(v)]
If $\cD$ is a bounded domain with outer cone property, then for every $f\in L^\infty(\cD)$ we have that $T^{\cD, V}_t f$ continuous
in $\bar\cD$ with value $0$ on the boundary, for every $t>0$.
\end{enumerate}
\end{lemma}

\begin{proof}
(i)\, \eqref{EL4.3A} follows from a standard conditioning argument, see \cite[Lem.~3.4]{HL-12}. We define
$$
T^{\cD, V}(t, x, y)=
\ex^0_{\Prob_S} \left[p_{S^\Psi_t}(x-y)
\ex^{x,y}_{0,S^\Psi_t}\left[e^{-\int_0^t V(B_{S^\Psi_s})ds}\Ind_{\{\uptau_{\cD} > t\}}\right]\right]\,,
$$
and show that
\begin{equation}\label{EL4.3B}
T^{\cD, V}(t, x, y)=T^{\cD, V}(t, y, x), \quad \text{for}\; t>0\,.
\end{equation}
Consider the Brownian bridge on the interval $[0, S^\Psi_t]$ starting with $x$ and ending at $y$ given by
$$
Z^{x, y}_s = (1-\frac{s}{S^\Psi_t}) x + \frac{s}{S^\Psi_t} y +  B_s - \frac{s}{S^\Psi_t} B_{S^\Psi_t}\,,
$$
where $\pro B$ is the Brownian motion running twice as fast as the standard Brownian motion, independent of the
subordinator $\pro {S^\Psi}$. A change of variable gives
$$
\int_0^t V(Z^{x, y}_{S^\Psi_s})\, \D{s} = \int_0^t V(Z^{x, y}_{S^\Psi_{t-s}})\, \D{s}\,,
$$
and we also have
$$
Z^{x, y}_{S^\Psi_s}\in \cD , \; \forall \, s\in[0, t] \;\;\Longleftrightarrow \;\;Z^{x, y}_{S^\Psi_{t-s}}\in \cD ,
\; \forall \, s\in[0, t].
$$
Therefore to show \eqref{EL4.3B} we only need to show that
\begin{equation}\label{EL4.3C}
\Bigl(Z^{x, y}_{S^\Psi_{t-\cdot}}\Big|_{[0, t]}, S^\Psi_t\Bigr) \Dst
\Bigl(Z^{y, x}_{S^\Psi_{\cdot}}\Big|_{[0, t]}, S^\Psi_t \Bigr).
\end{equation}
This can be shown by using the  fact that for any L\'evy process $\pro L$ starting at zero we have
\begin{equation}\label{EL4.3D}
(L_{t-\cdot}, L_t)\Dst (L_t- L_\cdot, L_t)\,.
\end{equation}
First we show \eqref{EL4.3C} using \eqref{EL4.3D}. Since the Brownian motion is independent of $\pro{S^\Psi}$, we get
the following equalities in distribution
\begin{align*}
Z^{x, y}_{S^\Psi_{t-\cdot}} & \Dst (1-\frac{S^\Psi_t -S^\Psi_\cdot}{S^\Psi_t}) x + \frac{S^\Psi_t -S^\Psi_\cdot}{S^\Psi_t} y
+  B_{S^\Psi_t -S^\Psi_\cdot} - \frac{S^\Psi_t -S^\Psi_\cdot}{S^\Psi_t} B_{S^\Psi_t}
\\
&\Dst \frac{S^\Psi_\cdot}{S^\Psi_t} x + (1-\frac{S^\Psi_\cdot}{S^\Psi_t}) y - B_{S^\Psi_\cdot} + \frac{S^\Psi_\cdot}{S^\Psi_t} B_{S^\Psi_t}
\\
&\Dst \frac{S^\Psi_\cdot}{S^\Psi_t} x + (1-\frac{S^\Psi_\cdot}{S^\Psi_t}) y + B_{S^\Psi_\cdot} - \frac{S^\Psi_\cdot}{S^\Psi_t} B_{S^\Psi_t}
= Z^{y, x}_{S^\Psi_{\cdot}}\,.
\end{align*}
This proves \eqref{EL4.3C}. Next we come to \eqref{EL4.3D}. It suffices to show that the finite dimensional distributions coincide.
Consider $t>s_1>s_2>\cdots>s_k\geq 0$ and $\xi_i\in\RR$ for $i=1, \ldots, k+1$. Then it is seen that
\begin{align*}
&\xi_1 L_{t-s_1} + \ldots + \xi_k L_{t-s_k} + \xi_{k+1}L_t
\\
&\quad = \sum_{i\geq 1}\xi_i L_{t-s_1} + \sum_{i\geq 2}^{k+1}\xi_i (L_{t-s_2} -L_{t-s_1}) + \cdots +
(\xi_k+\xi_{k+1}) (L_{t-s_{k}} -L_{t-s_{k-1}}) + \xi_{k+1}(L_{t} -L_{t-s_{k}})
\end{align*}
and
\begin{align*}
&\xi_1 (L_{t}-L_{s_1}) + \ldots + \xi_k (L_{t}-L_{s_k}) + \xi_{k+1} L_t
\\
&\quad = \sum_{i\geq 1}\xi_i (L_{t}-L_{s_1}) + \sum_{i\geq 2}^{k+1}\xi_i (L_{s_1} -L_{s_2}) + \cdots +
(\xi_k+\xi_{k+1}) (L_{s_{k+1}} -L_{s_{k}}) + \xi_{k+1} L_{s_k}\,.
\end{align*}
On the other hand,
$$
\Bigl( L_{t-s_1}, L_{t-s_2} -L_{t-s_1}, \cdots,  L_{t} -L_{t-s_{k}}\Bigr)\Dst
\Bigl(L_{t}-L_{s_1}, L_{s_1} -L_{s_2},\cdots,  L_{s_k}\Bigr)\,.
$$
Thus $(L_{t-s_1}, \cdots, L_{t-s_k}, L_t)$ has the same characteristic function as $(L_{t}-L_{s_1}, \cdots,L_{t}-L_{s_k}, L_t)$,
implying \eqref{EL4.3D}.

\medskip

(ii) We establish the Chapman-Kolmogorov relation
\begin{equation}\label{EL4.3E}
T^{\cD, V}(t+s, x, y) = \int_{\cD} T^{\cD, V}(t, x, u) T^{\cD, V}(s, u, y) \, \D{u},
\quad t, s>0, \, x, y\in\Rd\,.
\end{equation}
Denote
$$
\Xi(r, z, y)= \ex^{z, y}_{0, S^\Psi_r}\left[ e^{-\int_0^r V(Z_u)\, \D{u}} \Ind_{\{\uptau_D>r\}}\right],
$$
where $\pro Z$ denotes the Brownian bridge as defined above. Let $\pro{\widetilde{S}^\Psi}$ be a subordinator
given by Bernstein function $\Psi$, independent of $S^\Psi, B, Z$. Then we have
\begin{align*}
&\ex^0_{\Prob_S} \left[p_{S^\Psi_{t+s}}(x-y) \ex^{x,y}_{0,S^\Psi_{t+s}}\left[e^{-\int_0^{t+s}
V(Z_{S^\Psi_u})\D{u}}\Ind_{\{\uptau_{\cD} > t+s\}}\right]\right]
\\
&=
\ex^0_{\Prob_S} \left[p_{S^\Psi_{t+s}}(x-y) \ex^{x,y}_{0,S^\Psi_{t+s}}
\left[\Ind_{\{\uptau_{\cD}>t \}} e^{-\int_0^{t} V(Z_{S^\Psi_u})\D{u}}
\, \ex^{Z_{S^\Psi_t}, y}_{0, S^\Psi_{t+s}-S^\Psi_{t}}\left[e^{-\int_0^{s} V(Z_{S^\Psi_{u+t}-S^\Psi_t})\D{u}}
\Ind_{\{\uptau_{\cD} > s\}}\right]\right]\right]
\\
&=
\ex^0_{\Prob_S} \left[p_{S^\Psi_{t} + \widetilde{S}^\Psi_s}(x-y) \ex^{x,y}_{0,S^\Psi_{t} + \widetilde{S}^\Psi_s}
\left[\Ind_{\{\uptau_{\cD}>t \}} e^{-\int_0^{t} V(Z_{S^\Psi_u})\D{u}}
\, \ex^{Z_{S^\Psi_t}, y}_{0, \widetilde{S}^\Psi_s}\left[e^{-\int_0^{s} V(Z_{\widetilde{S}^\Psi_{u}})\D{u}}
\Ind_{\{\uptau_{\cD} > s\}}\right]\right]\right]
\\
&=
\ex^0_{\Prob_S} \left[p_{S^\Psi_{t} + \widetilde{S}^\Psi_s}(x-y) \ex^{x,y}_{0,S^\Psi_{t} +
\widetilde{S}^\Psi_s}\left[\Ind_{\{\uptau_{\cD} >t \}} e^{-\int_0^{t} V(Z_{S^\Psi_u})\D{u}} \,
\Xi(\widetilde{S}^\Psi_s, Z_{S^\Psi_t}, y)\right]\right]
\\
&=
\ex^0_{\Prob_S} \left[p_{S^\Psi_{t} + \widetilde{S}^\Psi_s}(x-y) \ex^{x}_{\Prob_W}\left[\Ind_{\{\uptau_{\cD}>t \}}
e^{-\int_0^{t} V(B_{S^\Psi_u})\D{u}} \, \Xi(\widetilde{S}^\Psi_s, B_{S^\Psi_t}, y)
p_{\widetilde{S}^\Psi_s}(B_{S^\Psi_t}-y)\frac{1}{p_{S^\Psi_{t} + \widetilde{S}^\Psi_s}(x-y)}\right]\right]
\\
&=
\ex^0_{\Prob_S} \left[ \ex^{x}_{\Prob_W}\left[\Ind_{\{\uptau_{\cD}>t \}} e^{-\int_0^{t} V(B_{S^\Psi_u})\D{u}}
\, \Xi(\widetilde{S}^\Psi_s, B_{S^\Psi_t}, y) p_{\widetilde{S}^\Psi_s}(B_{S^\Psi_t}-y)\right]\right]
\\
&=
\ex^0_{\Prob_S} \left[ \ex^{x}_{\Prob_W}\left[\Ind_{\{\uptau_{\cD}>t \}} e^{-\int_0^{t} V(B_{S^\Psi_u})\D{u}}
\, T^{\cD, V}(s, X_t, y)\right]\right]
\\
&=
\int_{\cD} T^{\cD, V}(t, x, u) T^{\cD, V}(s, u, y) \, \D{u},
\end{align*}
where the first equality follows from the Markov property of Brownian bridge, in the fourth line we used
\cite[Prop. ~A.1]{Sznitman}, and the sixth line follows by taking expectation with respect to $\pro{\widetilde{S}^\Psi}$.
Strong continuity follows along the line of \cite[Prop.~3.3]{Sznitman}.

\medskip

(iii) The symmetry of $T^{\cD, V}(t,x,y)$ implies that $T^{\cD, V}_t$ is a self-adjoint operator on $L^2(\cD)$. Let
$q_t(x, y)$ be the transition density of $\pro X$. Then the Hartman-Wintner condition \eqref{HW} implies that for every
$t>0$, $q_t(\cdot)$ is bounded and continuous \cite{HW-42,KS13} and therefore, $q_t(x, \cdot)\in L^2(\Rd)$. Indeed,
for $t>0$,
$$
q_{2t}(x, x)=\int_{\Rd} q_t(x-y) q_t(y-x)\, \D{y}= \int_{\Rd} q^2_t(x-y) \D{y}<\infty.
$$
The transition density for the process $\pro X$ killed upon the first exit from $\cD$ is given by Hunt's formula
$$
q^{\cD}_t(x, y)= q_t(x, y)-\ex^x\left[q_{t-\uptau_{\cD}}(X_{\uptau_{\cD}}, y)\Ind_{\{t>\uptau_{\cD}\}}\right]\,
\quad t>0, \; x, y\in\Rd\,.
$$
In particular, $q^{\cD}_t(x, y)\leq  q_t(x, y)\,$. Since $V$ is bounded, we obtain
\begin{align*}
\abs{T^{\cD, V}_t f(x)} &\leq e^{t \norm{V}_\infty} \int_{\cD} |f(y)| q^{\cD}_t(x, y)\D{y}
\leq e^{t \norm{V}_\infty} \norm{f}_{2,\cD} \, \norm{q_t(x, \cdot)}_{2}\,.
\end{align*}
Note that $\norm{q_t(x, \cdot)}_{2}$ does not depend on $x$. Therefore
$$
\int_{\cD \times \cD} (T^{\cD, V}(t, x, y))^2\, \D{y}\, \D{x}\leq C_t,
$$
with a constant $C_t > 0$, implying that $T^{\cD, V}_t$ is a Hilbert-Schmidt operator.

\medskip

(iv) We claim that for every $t>0$ and $y\in \cD$,
\begin{equation}\label{EL4.3F}
x \mapsto T^{\cD, V}(t, x, y) \quad \text{is continuous in}\; \cD.
\end{equation}
To show \eqref{EL4.3F}, write
\begin{align*}
T^{\cD, V}_\varepsilon(t, x, y)
&=
\ex^0_{\Prob_S}\left[\int_{\Rd}p_{S^\Psi_\varepsilon}(x-z) p_{\widetilde{S}^\Psi_{t-\varepsilon}}(z-y) \Xi(t-\varepsilon, z, y)\right]
\\
&=
\ex^0_{\Prob_S}\left[\int_{\cD}p_{S^\Psi_\varepsilon}(x-z) p_{\widetilde{S}^\Psi_{t-\varepsilon}}(z-y) \Xi(t-\varepsilon, z, y)\right]
\\
&=
\ex^0_{\Prob_S}\left[\int_{\Rd}p_{S^\Psi_\varepsilon+\widetilde{S}^\Psi_{t-\varepsilon}}(x-y)  \ex^{x, y}\left[e^{-\int_\varepsilon^t V(Z_u)\, \D{u}}
\Ind_{\{\uptau_{\cD} \circ\sigma_\varepsilon>t-\varepsilon\}}\right]\right]
\\
&= \ex^0_{\Prob_S}\left[\int_{\Rd}p_{S^\Psi_t}(x-y)  \ex^{x, y}\left[e^{-\int_\varepsilon^t V(Z_u)\, \D{u}}
\Ind_{\{\uptau_{\cD}\circ\sigma_\varepsilon>t-\varepsilon\}}\right]\right],
\end{align*}
where $\sigma_\varepsilon$ denotes the $\varepsilon$-shift operator, and the third line above follows from \cite[Cor.~A.2]{Sznitman}.
It is straightforward to see that \eqref{EL4.3F} holds for $T^{\cD, V}_\varepsilon$. On the other hand, $T^{\cD, V}_\varepsilon(t, \cdot, y)$
converges to $T^{\cD, V}(t, \cdot, y)$ as $\varepsilon\to 0$, uniformly on the compact subsets of $\cD$, see for example,
\cite[eq. (3.21)]{Sznitman}. This proves \eqref{EL4.3F}. The proof of (iv) can be completed employing a similar argument as in
\cite[Prop.~3.5]{Sznitman} combining \eqref{EL4.3E}, \eqref{EL4.3F} and (ii).

\medskip

Finally we prove (v). Denote $\tilde f(t,x)=T^{\cD, V}_t f(x)$. In view of (iv) it is enough to show that
for $x_n\to z\in\partial \cD$ we have
\begin{equation}\label{EL4.3G}
\lim_{n\to\infty}\abs{\tilde f(t,x_n)}=0.
\end{equation}
Since $f$ and $V$ are bounded, we obtain from \eqref{killedV} that
$$
\abs{\tilde f(t,x_n)}\leq e^{\norm{V}_\infty t} \norm{f}_\infty \Prob^{x_n} (\uptau_{\cD} >t).
$$
Since $z\in\cD$ is regular, see the proof of \cite[Lem.~2.9]{BGR15}, we have
$$
\lim_{x_n\to z}\Prob^{x_n} (\uptau_{\cD}>t)=0.
$$
By combining the above two equalities \eqref{EL4.3G} follows.
\end{proof}

\begin{remark}\label{R3.1}
We note that Lemma \ref{FKsemi} can be obtained also for $\Psi$-Kato class potentials, which may have local
singularities (see below).
Also, further (such as contractivity, positivity improving etc) properties of $\semi {T^{\cD, V}}$ can be shown, which
is left to the reader. The lemma can further be extended for other non-local Schr\"odinger operators, involving more
general isotropic L\'evy processes.
\end{remark}

From Lemma \ref{FKsemi} it then follows that the Dirichlet-Schr\"odinger eigenvalue equation \eqref{dirichev} is
solved by a countable set of eigenvalues $\lambda_1^{\cD,V} < \lambda_2^{\cD,V} \leq \lambda_3^{\cD,V}\leq \cdots
\to \infty$ and a corresponding orthonormal set of $L^2(\cD)$-eigenfunctions, such that the principal eigenvalue
is simple and the corresponding principal eigenfunction has a strictly positive version.

\subsection{The location of extrema}
In the remaining part of this article we shall assume that $\cD$ is a bounded, convex set.
In the following we will use a class of potentials, which are general enough to contain many interesting cases
(such as Coulomb-type potentials), while being naturally suitable for defining Feynman-Kac semigroups. Consider
the set of functions
\begin{align}
\label{eq:Katoclass}
\sK^\Psi = \Big\{f: \RR\to \Rd: \, \mbox{$f$ is Borel measurable and} \,
\lim_{t \downarrow 0} \sup_{x \in \Rd} \ex^x \Big[\int_0^t |f(X_s)| ds\Big] = 0\Big\}.
\end{align}
We say that the potential $V: \Rd \to \RR$ belongs to $\Psi$-Kato class whenever it satisfies
$$
\quad V_- \in \sK^\Psi \quad \text{and} \quad V_+ \in \sK^\Psi_{\rm loc}, \quad \text{with} \quad V_+ =
\max\{V,0\}, \; V_- = \min\{V,0\},
$$
where $V_+ \in \sK^\Psi_{\rm loc}$ means that $V_+ 1_\cC \in \sK^\Psi$ for all compact sets $\cC \subset \Rd$,
and $\pro X$ is the L\'evy process generated by $\Psi(-\Delta)$. It is direct to see that $L^{\infty}_{\rm loc}(\Rd)
\subset \sK_{\rm loc}^\Psi$, moreover, by stochastic continuity of $\pro X$ also $\sK_{\rm loc}^\Psi \subset
L^1_{\rm loc}(\Rd)$. By standard arguments based on Khasminskii's Lemma, for a $\Psi$-Kato class potential $V$ it
follows that there exist suitable constants $C_1(\Psi,V), C_2(\Psi,V) > 0$ such that
\begin{align}
\label{khasmin}
\sup_{x \in \Rd} \ex^x\left[e^{-\int_0^t V(X_s)\D{s}}\right] \leq \sup_{x \in \Rd}
\ex^x\left[e^{\int_0^t V_-(X_s)\D{s}}\right] \leq C_1 e^{C_2t}, \quad t>0.
\end{align}
For further details we refer to \cite[Sect. 4]{HIL12} and \cite{LHB}.

For Bernstein functions we will use the following property repeatedly below, which has been introduced in \cite{BGR14b}.
\begin{assumption}
\label{WLSC}
The function is said to satisfy a weak local scaling (WLSC) property with parameters $\mu > 0$ and
$\underline{c}\in(0, 1]$, if
$$
\Psi(\gamma u) \;\geq\; \underline{c}\, \gamma^\mu \Psi(u), \quad u>0, \; \gamma\geq 1.
$$
\end{assumption}
\noindent
We will show some typical examples of Bernstein functions satisfying Assumption \ref{WLSC} further below in this
section.

Now we present two expressions of the main result of this section. The first uses $\Psi$-Kato class potentials $V$ and a
restricted class of $\Psi$, the second uses a more general class of Bernstein functions $\Psi$ and bounded potentials.

\begin{theorem}\label{T3.20}
Let $\Psi\in \mathcal B_0$ satisfy Assumption \ref{WLSC} with $\mu > 0$ and $\underline{c}\in(0, 1]$. Let $V \in
\sK^\Psi$ be a $\Psi$-Kato class potential, with $V^-\in L^p(\Rd)$, $p>\frac{d}{2\mu}$. Also, let $\varphi$ be a
non-zero solution of \eqref{dirichev} at eigenvalue $\lambda^{V,\cD}$. Assume that $|\varphi|$ attains a global maximum
at $x^*\in\cD$, and denote $r =\dist(x^*, \partial\cD)$ and $\eta=1-\frac{d}{2\mu p}$. Then there exists a constant
$\Theta_1 > 0$, dependent on $d$, $\mu$, $\underline{c}$, $\eta$, $\inrad \cD$, and a constant $\Theta_2 > 0$,
dependent on $\eta$ only, such that
\begin{equation}\label{ET3.2A}
\Theta_1\norm{V^-}^{\nicefrac{1}{\eta}}_{p}-\inf_{\cD}V^+ + \lambda^{V,\cD} \geq \Theta_2 \,\Psi(r^{-2})\,.
\end{equation}
\end{theorem}

The proof of Theorem~\ref{T3.20} is simpler if the potential $V$ is bounded. Moreover, one can allow a larger class of
$\Psi$, not necessarily satisfying WLSC, and the dependence of $\Theta_1$ on the domain parameters can be waived when
$V \in L^\infty(\cD)$. This is obtained in the following theorem.
\begin{theorem}\label{T3.1}
Let $\Psi \in \mathcal B_0$, $V \in L^\infty(\Rd)$, and $\varphi$ be a non-zero solution of \eqref{dirichev} at
eigenvalue $\lambda^{V,\cD}$. Assume that $|\varphi|$ attains a global maximum at $x^*\in\cD$, and denote $r =
\dist(x^*, \partial\cD)$. Then there exists a universal constant $\theta > 0$, independent of $\cD$, $x^*$, $V$,
$\Psi$ and the dimension $d$, such that
\begin{equation}\label{ET3.1A}
\norm{V^-}_{\infty, \cD}-\inf_{\cD}V^+ + \lambda^{V,\cD} \geq \theta \,\Psi(r^{-2}),
\end{equation}
with
\begin{equation}
\label{theta}
\theta = -\min_{\kappa > 1}\frac{1}{\kappa}\log\left(1-F(-1)(1-e^{1-\kappa})\right) \approx 0.0833,
\end{equation}
where $F$ is the probability distribution function of a Gaussian random variable $N(0,2)$. In particular,
if $\Psi$ is strictly increasing, then
\begin{equation}
\label{inv}
\dist(x^*, \partial\cD) \geq \frac{1}{\sqrt{\Psi^{-1}\left(\frac{\norm{V^-}_{\infty, \cD}-\inf_{\cD}V^+ +\lambda^{V,\cD}}{\theta}\right)}}.
\end{equation}
\end{theorem}

Next we turn to proving these theorems. For technical reasons we start by showing first the latter theorem.
\begin{proof}[Proof of Theorem~\ref{T3.1}]
Let $\uptau_\cD$ be the first exit time of $\pro X$ from $\cD$, as defined in \eqref{exit}. Using the eigenvalue
equation and the representation (\ref{killedV}), we have that
$$
|\varphi(x^*)| \leq e^{\lambda^{V,\cD} \, t }\ex^{x^*} [e^{-\int_0^{t} V(X_s) \D s}
|\varphi(X_{t})|\Ind_{\{t<\uptau_\cD\}}] \leq |\varphi(x^*)| e^{\lambda^{V,\cD} t} e^{(\norm{V^-}_{\infty, \cD}-\inf_{\cD}V^+) t}
\Prob^{x^*}(\uptau_\cD > t),
$$
that is,
\begin{equation}\label{E1.2}
e^{t\, (\norm{V^-}_{\infty, \cD}-\inf_{\cD}V^+ +\lambda^{V,\cD})} \Prob^{x^*}(\uptau_\cD > t) \;\geq \; 1\,, \quad t\geq 0.
\end{equation}
We choose
\begin{equation}
\label{chooset}
t = \frac{\kappa}{\Psi(r^{-2})}
\end{equation}
with a suitable $\kappa$, which will be justified below, and show that for this $t$ we have
\begin{equation}
\label{smalldelta}
\Prob^{x^*}(\uptau_\cD>t)< \delta <1,
\end{equation}
where $\delta$ does not depend on $x^*$, $\cD$.

Let $z\in\partial \cD$ be such that $\dist(x^*, z)=r$, and consider the half-space $\cH \subset \cD^c$ intersecting
$\cD$ at $z$. Note that this is made possible by the convexity of $\cD$, and
$$
\Prob^{x^*}(\uptau_\cD\leq t) \geq \Prob^{x^*}(X_t\in \cH)
$$
holds. We assume with no loss of generality that $\cH$ is perpendicular to the $x$-axis, $x^*=0$ and $z=(r, 0,\ldots, 0)$.
This is possible, since we can inscribe a ball of radius $r$ in $\bar{\cD}$ centered at $x^*$ and $\cH$ would be a tangent plane to it at the point $z$.
Therefore, we have for $s\geq r^2$ that
\begin{equation}
\label{below}
\Prob^{x^*}_W(B_s\in \cH)
=\Prob^0_W(B^1_s \geq r) = \frac{1}{\sqrt{4\pi}}\int_{\frac{r}{\sqrt{s}}}^\infty e^{-\frac{y^2}{4}}\, \D{y}
\geq \frac{1}{\sqrt{4\pi}}\int_{1}^\infty e^{-\frac{y^2}{4}}\, \D{y} = 
F(-1)\,,
\end{equation}
where $\pro {B^1}$ denotes a one-dimensional Brownian motion running twice as fast as standard Brownian motion, and $F$
is the probability distribution function of a Gaussian random variable with mean $0$ and variance $2$.
Using the subordination formula \eqref{subord} and the uniform estimate (\ref{below}), we have
\begin{eqnarray*}
\Prob^{x^*}(X_t\in \cH)
&=&
\int_0^\infty \Prob^{x^*}_W(B_s\in \cH) \mathbb P_S (S^\Psi_t \in \D s) \\
&\geq&
\int_{r^2}^\infty  \Prob^{x^*}_W(B_s\in \cH) \mathbb P_S (S^\Psi_t \in \D s) \geq F(-1) \Prob_S (S^\Psi_t\geq r^2).
\end{eqnarray*}
By (\ref{lapla}) and (\ref{chooset}) we have
$$
\Prob_S (S^\Psi_t\leq r^2) = \Prob_S (e^{-r^{-2}S^\Psi_t}\geq  e^{-1})\leq e\, \Exp_{\Prob_S}[e^{-r^{-2}S^\Psi_t}] =
e^{1-t\Psi(r^{-2})} = e^{1-\kappa}.
$$
Hence with $\kappa > 1$ we obtain $\Prob_S (S^\Psi_t\leq r^2)<1$, and thus \eqref{smalldelta} holds with $\delta =
1-F(-1)(1-e^{1-\kappa})$, independently on $r$. This then implies \eqref{ET3.1A} with constant prefactor
$$
\theta_\kappa = -\frac{1}{\kappa}\log\left(1-F(-1)(1-e^{1-\kappa})\right)\,,
$$
which on optimizing over $\kappa$ gives the constant \eqref{theta}.
\end{proof}

\begin{proof}[Proof of Theorem~\ref{T3.20}]
The key estimate for the proof is the following improvement of \eqref{khasmin}: for any $\kappa_1>0$ there exists a
constant $C_{1}>0$, dependent on $\kappa_1, d, \mu, \underline{c}$, satisfying for $t\in [0, \kappa_1]$ and $\vartheta>0$
\begin{equation}\label{ET3.2B}
\sup_{x\in\Rd}\ex^x\left[e^{\int_0^t \vartheta V^-(X_s)\, \D{s}}\right]\leq \;
m_\eta e^{\bigl(C_1\vartheta\norm{V^-}_{p}\Gamma(\eta)\bigr)^{\nicefrac{1}{\eta}} t}\,,
\end{equation}
where $\eta=1-\frac{d}{2\mu p}$ and $m_\eta$ depends only on $\eta$. First we complete the proof of the theorem assuming
\eqref{ET3.2B}.

Choose $\kappa_1 = \frac{2}{\Psi([\inrad \cD]^{-2})}$. Suppose that $r=\dist(x^*, \cD)$ and let $t=\frac{2}{\Psi(r^{-2})}\leq \kappa_1$.
Then using \eqref{killedV} and H\"{o}lder inequality, we obtain for $\vartheta\geq 1$ that
\begin{equation}\label{ET3.2C}
1\;\leq\;
e^{t\,(\lambda^{V, \cD}-\inf_{\cD} V^+)}\; \ex^{x^*}\left[e^{\int_0^t \vartheta V^-(X_s)\, \D{s}}\right]^{\nicefrac{1}{\vartheta}}
\;\left(\Prob^{x^*}(\uptau_{\cD}> t)\right)^{\frac{\vartheta-1}{\vartheta}}\,.
\end{equation}
Hence from \eqref{smalldelta}, \eqref{ET3.2B} and \eqref{ET3.2C} we see that
\begin{align*}
1&\leq \delta^{\frac{\vartheta-1}{\vartheta}} (m_\eta)^{\nicefrac{1}{\vartheta}}
\exp\left(t\left[\lambda^{V, \cD}-\inf_{\cD} V^+
+ \frac{1}{\vartheta}(C_1\vartheta\norm{V^-}_{p}\Gamma(\eta)\bigr)^{\nicefrac{1}{\eta}}\right]\right)
\\
& =\delta  \left(\frac{m_\eta}{\delta}\right)^{\nicefrac{1}{\vartheta}}
\exp\left(t\left[\lambda^{V, \cD}-\inf_{\cD} V^+
+ \frac{1}{\vartheta}(C_1\vartheta\norm{V^-}_{p}\Gamma(\eta)\bigr)^{\nicefrac{1}{\eta}}\right]\right)\,.
\end{align*}
Since $\delta<1$ and $\lim_{\vartheta\to\infty} \left(\frac{m_\eta}{\delta}\right)^{\nicefrac{1}{\vartheta}}=1$, we can choose
$\vartheta$ large enough such that
$$
\delta_1= \delta  \left(\frac{m_\eta}{\delta}\right)^{\nicefrac{1}{\vartheta}}\;<\;1.
$$
Thus we obtain
$$
\log \frac{1}{\delta_1} \leq t \left(\lambda^{V, \cD}-\inf_{\cD} V^+
+ \frac{1}{\vartheta}(C_1\vartheta\norm{V^-}_{p}\Gamma(\eta)\bigr)^{\nicefrac{1}{\eta}}\right),
$$
implying
$$
\left(\frac{1}{2}\log\frac{1}{\delta_1}\right)\, \Psi(r^{-2})\;\leq\; \lambda^{V, \cD}-\inf_{\cD} V^+ + \frac{1}{\vartheta}(C_1\vartheta\norm{V^-}_{p}\Gamma(\eta)\bigr)^{\nicefrac{1}{\eta}}.
$$
This gives \eqref{ET3.2A} for
$$
\Theta_1 = \frac{1}{\vartheta}(C_1\vartheta \Gamma(\eta)\bigr)^{\nicefrac{1}{\eta}} \quad \text{and}
\quad \Theta_2= \frac{1}{2}\log\frac{1}{\delta_1}\,.
$$

Now we proceed to establish \eqref{ET3.2B}. Since $\Psi$ has the WLSC property, the characteristic exponent
$\Phi(r)=\Psi(r^2)$ also has the WLSC property, namely
$$
\Phi(\gamma u) \;\geq\; \underline{c}\, \gamma^{2\mu} \Phi(u), \quad \text{for all} \; u>0
\; \text{and}\; \gamma\geq 1.
$$
Thus by \cite[Prop.~19]{BGR14b} there exists a constant $K_1$, dependent on $d, \mu, \underline{c}$, satisfying
\begin{equation}\label{ET3.2D}
q_t(x, y)= q_t(\abs{x-y})\leq K_1\left(\Phi^{-1}\left(\frac{1}{t}\right)\right)^d, \quad \forall\; t>0\,.
\end{equation}
Here $q_t(x,y)$ denotes the transition density function of $\pro X$. On the other hand, from the WLSC property
of $\Phi$ it follows that
$$
\Phi^{-1}(\lambda)\leq \lambda^{\frac{1}{2\mu}}\frac{u}{\Phi(u)^{\frac{1}{2\mu}}}\,
\quad \text{for all}\; \lambda\geq \Phi(u), \; u>0.
$$
Choose $\nu>0$ and denote $\nu_1=\frac{\nu}{(\Phi(\nu))^{\frac{1}{2\mu}}}$. Then for $s\geq \Phi(\nu)$ we
obtain
\begin{equation}\label{ET3.20dep}
\Phi^{-1}(s)\leq \nu_1\,s^{-2\mu}.
\end{equation}
Hence, using the above estimate in \eqref{ET3.2D} we get that
\begin{equation}\label{ET3.2E}
q_t(x, y)\leq K_2\, t^{-\frac{d}{2\mu}}, \quad t\leq \frac{1}{\Psi(\nu^2)},
\end{equation}
where $K_2$ depends on $d, \mu$ and $\nu_1$. Let $\kappa_1$ be positive and choose $\nu=
\sqrt{\Psi^{-1}(\frac{1}{\kappa_1})}$. With this choice of $\nu$ we have from \eqref{ET3.2E} that
\begin{equation}\label{ET3.2F}
q_t(x, y)\leq K_2\, t^{-\frac{d}{2\mu}}, \quad t\leq \kappa_1.
\end{equation}
For every $t\in (0, \kappa_1]$ and $f\in L^p(\Rd)$ we have
\begin{align*}
\ex^x\left[f(X_t)\right] &\leq \norm{f}_{p} \left[\int_{\Rd} (q_t(\abs{x-y}))^{p'}\, \D{y}\right]^{\nicefrac{1}{p'}}
\\
&\leq
K_2^{\nicefrac{1}{p}} \norm{f}_{p}  t^{-\frac{d}{2\mu p}}\left[\int_{\Rd} q_t(\abs{x-y})\, \D{y}\right]^{\nicefrac{1}{p'}}
 =
K_3 \norm{f}_{p}  t^{-\frac{d}{2\mu p}},
\end{align*}
where $p'=\frac{p}{p-1}$, $K_3=K_2^{\nicefrac{1}{p}}$, and in the second line above we used \eqref{ET3.2F}.
Let now $0 \leq s_1 \leq \ldots \leq s_{k}$, $k \in \mathbb N$. Using the Markov property of $\pro X$ with respect to
its natural filtration $\pro\cF$, for $f\geq 0$ we obtain
\begin{eqnarray*}
\ex^x[f(X_{s_1})\cdots f(X_{s_k})]
&=&
\ex^x[f(X_{s_1})\cdots f(X_{s_{k-1}}) \ex^x[f(X_{s_k})\,|\, \cF_{s_{k-1}}]] \\
&=&
\ex^x[f(X_{s_1})\cdots f(X_{s_{k-1}}) \ex^{X_{s_k}}[f(X_{s_k-s_{k-1}})] \\
&\leq&
K_3 \norm{f}_{p} (s_k-s_{k-1})^{-\frac{d}{2\mu p}} \ex^x[f(X_{s_1})\cdots f(X_{s_{k-1}})] \\
&\leq &
\!\!\! \ldots \leq (K_3 \norm{f}_{p})^k s_1^{-\frac{d}{2\mu p}} (s_2-s_1)^{-\frac{d}{2\mu p}}
\cdots (s_k - s_{k-1})^{-\frac{d}{2\mu p}}.
\end{eqnarray*}
Hence (compare \cite[Lem. 4.51]{LHB} in the second edition)
\begin{eqnarray*}
\lefteqn{
\ex^x\left[\frac{1}{k!}\left(\int_0^t f(X_s)\, \D{s}\right)^k\right] } \\
&\leq \;&
\int_0^t \D{s_1} \int_{s_1}^t \D{s_2} ... \int_{s_k}^t \D{s_k} \, \ex^x [f(X_{s_1})f(X_{s_2})...f(X_{s_k})] \\
&\leq \;&
K_3^k \norm{f}_{p}^k \int_0^t \D{s_1}\int_{s_1}^t \D{s_2} ... \int_{s_k}^t \D{s_k} \,
s_1^{-\frac{d}{2\mu p}} (s_2-s_1)^{-\frac{d}{2\mu p}} \cdots (s_k - s_{k-1})^{-\frac{d}{2\mu p}} \\
&=&
\frac{\left(K_3\norm{f}_{p}t^\eta\Gamma(\eta)\right)^k}{\Gamma(1+k\eta)}, \quad t\leq \kappa_1,
\end{eqnarray*}
where $\eta=1-\frac{d}{2\mu p}>0$, by our choice of $p$. Recall the Mittag-Leffler function
$$
\sM_{\beta}(x)=\sum_{k=0}^\infty \frac{ x^k}{\Gamma(1+\beta k)}
$$
(see \cite{GKMR} for definitions and properties). We find by the above that for $t\in[0, \kappa_1]$,
\begin{equation}\label{ET3.2G}
\sup_{x\in\Rd}\ex^x\left[e^{\int_0^tf(X_s)\, \D{s}}\right]\leq \sM_\eta(K_3\norm{f}_{p}t^\eta\Gamma(\eta)).
\end{equation}
It is also known that for some constant $m_\eta$, dependent only on $\eta$,
$$
\sM_\eta(x)\leq m_\eta e^{x^{\nicefrac{1}{\eta}}}, \quad  x\geq 0,
$$
holds. Thus, using \eqref{ET3.2G} we have for $t\leq \kappa_1$ that
\begin{equation}\label{ET3.2H}
\sup_{x\in\Rd}\ex^x\left[e^{\int_0^tf(X_s)\, \D{s}}\right]
\leq m_\eta e^{\left(K_3\norm{f}_{p}\Gamma(\eta)\right)^{\frac{1}{\eta}} t} .
\end{equation}
Putting $f=\vartheta V^-$ in \eqref{ET3.2H}, we obtain \eqref{ET3.2B}.
\end{proof}

\medskip
The dependence of $\Theta_1$ on $\inrad \cD$ is due to the factor
$$
\nu_1=\frac{\nu}{(\Phi(\nu))^{\frac{1}{2\mu}}},
$$
which appears in \eqref{ET3.20dep}. For $\Psi(u)=u^{\nicefrac{\alpha}{2}}$, however, $\nu_1$ does not depend on $\nu$.
Thus we have the following improvement to Theorem~\ref{T3.20}.
\begin{corollary}
Suppose that $\Psi(u)=u^{\nicefrac{\alpha}{2}}$. Moreover, assume that $V$ is a $\Psi$-Kato class function with
$V^-\in L^p(\Rd)$, $p>\frac{d}{\alpha}$. Let  $\varphi$ be a non-zero solution of
\eqref{dirichev} at eigenvalue $\lambda^{V,\cD}$. Assume that $|\varphi|$ attains a global maximum at $x^*\in\cD$,
and denote $r =\dist(x^*, \partial\cD)$. Then there exist $\Theta_1$, dependent on $d, \alpha,\underline{c}, \eta $,
and $\Theta_2$, dependent on $\eta$, where $\eta=1-\frac{d}{\alpha p}$, such that
\begin{equation*}
\Theta_1\norm{V^-}^{\nicefrac{1}{\eta}}_{p}-\inf_{\cD}V^+ + \lambda^{V,\cD} \geq \Theta_2 \,\Psi(r^{-2})
\end{equation*}
holds.
\end{corollary}

\begin{remark}
For classical Schr\"odinger operators we have $\Psi(u) = u$, for which Theorem \ref{T3.1} implies (\ref{RS}), possibly
with a different constant $c$. Also, for fractional Schr\"odinger operators we have $\Psi(u) = u^{\alpha/2}$, which
reproduces the result obtained in \cite{Biswas-17}. Formulae \eqref{ET3.1A}-\eqref{inv} equally apply for $V \equiv 0$,
in which case the statement refers to the Dirichlet eigenfunctions and eigenvalues.
\end{remark}

\begin{example}
Some important examples of $\Psi$ satisfying Assumption \ref{WLSC} include:
\begin{itemize}
\item[(i)]
$\Psi(u)=u^{\alpha/2}, \, \alpha\in(0, 2]$, with $\mu = \frac{\alpha}{2}$.
\item[(ii)]
$\Psi(u)=(u+m^{2/\alpha})^{\alpha/2}-m$, $m> 0$, $\alpha\in (0, 2)$, with $\mu = \frac{\alpha}{2}$.
\item[(iii)]
$\Psi(u)=u^{\alpha/2} + u^{\beta/2}, \, \alpha, \beta \in(0, 2]$, with $\mu = \frac{\alpha}{2}
\wedge \frac{\beta}{2}$.
\item[(iv)]
$\Psi(u)=u^{\alpha/2}(\log(1+u))^{\beta/2}$, $\alpha \in (0,2)$, $\beta \in (0, 2-\alpha)$, with $\mu=\frac{\alpha}{2}$.
\item[(v)]
$\Psi(u)=u^{\alpha/2}(\log(1+u))^{-\beta/2}$, $\alpha \in (0,2]$, $\beta \in [0,\alpha)$ with $\mu=\frac{\alpha-\beta}{2}$.
(Since for $\gamma\geq 1$, $u>0$, $(1+u)^{\gamma}\geq (1+\gamma u)$ holds, we have $\gamma^{\nicefrac{\beta}{2}}
(\log (1+u))^{\nicefrac{\beta}{2}}\geq (\log (1+\gamma u))^{\nicefrac{\beta}{2}}$.)
\end{itemize}
\end{example}

\subsection{Consequences on the spectrum}
The above theorems have a number of implications on the eigenvalues and related quantities. Here we discuss these
implications involving an interplay of survival times of paths and geometric features.
\begin{corollary}\label{C3.1}
Let $\varphi$ be an eigenfunction corresponding to eigenvalue $\lambda^{V,\cD}$ of $\Psidel + V$ under the conditions
of Theorem~\ref{T3.1}. Suppose that $\lambda^{V,\cD}>0$. Then we have
\begin{equation}\label{C3.1A}
\int_0^\infty \ex^{x^*}\left[ e^{-\int_0^t V(X_s)\, \D{s}} \Ind_{\{\uptau_\cD>t\}}\right]\, \D{t}\; \geq \;
\frac{1}{\lambda^{V,\cD}}\,,
\end{equation}
where $x^*$ is a maximizer of $\abs{\varphi}$ in $\cD$.
\end{corollary}
\begin{proof}
From the proof of Theorem~\ref{T3.1} we have
\begin{equation}\label{EC3.2A}
\ex^{x^*}\left[ e^{-\int_0^t V(X_s)\, \D{s}} \Ind_{\{\uptau_\cD>t\}}\right] \;\geq \; e^{-\lambda^{V,\cD}t}\,, \quad t\geq 0\,.
\end{equation}
By integrating both sides in $t$ on $(0, \infty)$, we obtain \eqref{C3.1A}.
\end{proof}
\noindent
Note that the left hand side of \eqref{C3.1A} gives the mean survival time of the process $\pro X$ starting from $x^*$,
perturbed by the potential $V$, thus the above result gives a probabilistic bound on the Dirichlet-Schr\"odinger eigenvalues.

\begin{remark}
Using the trivial bound
$$
\dist(x^*, \partial\cD) \leq \inrad \cD,
$$
involving the inradius of $\cD$, we get the geometric constraint
$$
\lambda^{V,\cD} \geq \theta \Psi\left( \frac{1}{(\inrad \cD)^2}\right) - \norm{V^-}_\infty + \inf_{\cD} V^+
$$
on the bottom of the spectrum.
\end{remark}

\begin{remark}
\label{doob}
Since $\Psi^{-1}$ is an increasing function, the bound \eqref{inv} can be interpreted as saying that if the potential
is not strong enough, the global extrema of $\varphi$ cannot be too close to the boundary. Intuitively it is clear that
one can decrease $\dist(x^*, \partial\cD)$, for instance, by a potential which has a hole close to the boundary, that
is deep enough to make the process stay in that region with a sufficiently high probability, preventing it to hit the
boundary too soon and get killed. It is seen that the condition only requires sufficient strength of the potential and
no details on its local behaviour.
There is also a probabilistic interpretation of relation \eqref{ET3.1A}. From \cite[Rem.~4.8]{Schilling-98} we find
that
$$
c_1 \Exp^{x}[\uptau_{\sB_r(x)}]\;\leq \frac{1}{\Psi(r^{-2})}\;\leq c_2 \Exp^x[\uptau_{\sB_r(x)}]\,,
$$
for some constants $c_1, c_2 > 0$ depending only on $d$. A combination with \eqref{C3.1A} then implies that the inequality
makes a comparison of the mean survival time of the process starting from $x^*$ perturbed by the potential with the mean
survival time of the free (unperturbed) process, involving the proportionality constant $\theta$. Note that since the
principal eigenfunction $\varphi_1$ is strictly positive, by the Doob $h$-transform $f \mapsto \varphi_1 f$, $f \in L^2(\cD)$,
we can construct a random process generated by the operator $\widetilde H f = \frac{1}{\varphi_1} H(\varphi_1 f)$ whose
stationary measure is $\varphi_1^2 dx$. The location $x^*$ of a global maximum of $\varphi_1$ then corresponds to a mode
of the stationary density of the process conditioned never to exit the domain $\cD$.
\end{remark}

Next we consider the principal Dirichlet eigenvalues in the absence of a potential.
\begin{corollary}\label{C3.2}
Let $V=0$ and consider the principal eigenvalue $\lambda_1^\cD$ of the Dirichlet problem \eqref{dirichev} for $\Psi(-\Delta)$.
\begin{enumerate}
\item[(i)]
We have
\begin{equation}\label{C3.1B}
\lambda_1^\cD \geq \left(\frac{\Gamma(p+1)}{\sup_{x\in \cD} \ex^{x}[\uptau^p_{\cD}]}\right)^{1/p},
\end{equation}
for every $p \geq 1$.

\item[(ii)]
Let $\Psi \in \mathcal B_0$ be a complete Bernstein function. Then there exist positive universal constants $C_1, C_2$,
dependent only on $d$, such that
\begin{equation}\label{C3.1C}
\frac{C_1}{\Psi([\inrad \cD]^{-2})}\;\leq\;\sup_{x\in \cD} \ex^{x}[\uptau_{\cD}]\;\leq\; \frac{C_2}{\Psi([\inrad \cD]^{-2})}.
\end{equation}

\item[(iii)]
 There exists a constant $C_3 > 0$, dependent on $d$, such that
\begin{equation}\label{C3.1C0}
\frac{1}{\sup_{x\in \cD} \ex^{x}[\uptau_{\cD}]}\;\leq \lambda_1^\cD \leq \frac{C_3}{\sup_{x\in \cD} \ex^{x}[\uptau_{\cD}]}\,.
\end{equation}
\end{enumerate}
\end{corollary}

\begin{proof}
Let $p \geq 1$. To obtain (i) multiply both sides of \eqref{EC3.2A} by $p t^{p-1}$ and integrate with respect to $t$
over $(0,\infty)$.

Next consider (ii). To prove \eqref{C3.1C} first note that by the domain monotonicity property we have
$$
\lambda_{1, \mathrm{Lap}}^\cD \leq \frac{\kappa_1}{[\inrad \cD]^2}\,,
$$
where $\kappa_1=\lambda_{1, \mathrm{Lap}}^\sB$ is the Dirichlet principal eigenvalue in the unit ball, and
$\lambda_{1, \mathrm{Lap}}^\cD$ denotes the principal Dirichlet eigenvalue of the Laplacian in $\cD$. Therefore
by \cite{Chen-Song} we obtain
$$
\lambda_1^{\cD}  \leq \Psi(\lambda_{1, \mathrm{Lap}}^\cD)\leq \Psi\left(\frac{\kappa_1}{[\inrad \cD]^2}\right).
$$
Thus, using \eqref{C3.1B} for $p=1$, we have
$$
\frac{1}{\Psi(\kappa_1^{-1}[\inrad \cD]^{-2})}\;\leq\;\sup_{x\in \cD} \ex^{x}[\uptau_{\cD}]\,.
$$
From the Laplace transform of $\pro {S^\Psi}$ and the monotonicity of $\Psi$ it is seen that for every
$\delta\geq 1$ we have
\begin{equation}\label{Psi-mon}
\Psi(u)\leq \Psi(\delta u)\leq \delta \Psi(u), \quad \forall\; u\geq 0\,.
\end{equation}
Thus by \eqref{Psi-mon} we get the left hand side of \eqref{C3.1C} with $\kappa^{-1}_1\vee 1=C^{-1}_1$.
To prove the converse implication we use a
result from \cite{Mendez}. Note that since $\Psi$ is a complete Bernstein function, the process $\pro X$ has a
transition density $q(t, x, y) = q(t, x-y)$. Moreover, $q(t, \cdot)$ is radially symmetric and decreasing. Denote
by $r_\cD=\inrad \cD$ and define
$$
S_\cD= \left\{x\in\Rd\; :\; x\in \RR^{d-1}\times(-r_\cD, r_\cD)\right\}\,.
$$
Fix $t>0$ and $z_0\in\cD$. By $\uptau_{S_\cD}$ we denote the first exit time from $S_\cD$. Then
\begin{align*}
\Prob^{z_0}(\uptau_\cD>t) & = \lim_{m\to\infty} \Prob^{z_0}\left(X_{\frac{t}{m}}\in \cD, X_{\frac{2t}{m}}\in \cD, \ldots,
X_{\frac{mt}{m}}\in \cD\right)
\\
&= \lim_{m\to\infty} \int_{\cD}\int_{\cD}\cdots\int_{\cD} \Pi_{j=1}^m q(\frac{t}{m}, z_j-z_{j-1}) \D{z_1}\D{z_2}\cdots\D{z_m}
\\
&\leq  \lim_{m\to\infty} \int_{S_\cD}\int_{S_\cD}\cdots\int_{S_\cD} q(\frac{t}{m}, 0, z_1) \Pi_{j=2}^m q(\frac{t}{m}, z_j-z_{j-1})
\D{z_1}\D{z_2}\cdots\D{z_m}
\\
&= \lim_{m\to\infty} \Prob^{0}\left(X_{\frac{t}{m}}\in S_\cD, X_{\frac{2t}{m}}\in S_\cD, \ldots, X_{\frac{mt}{m}}\in S_\cD\right)
=\Prob^{0}(\uptau_{S_\cD}>t),
\end{align*}
where in the inequality above we used \cite[Th.~1.2]{Mendez}. On the other hand, the first exit time $\uptau_{S_\cD}$ starting from
$0$ is equal in distribution to the first exit time of a one-dimensional subordinate Brownian motion from the interval $\sB_{r_\cD}
= (-r_\cD, r_\cD)$ starting  from $0$. Let $(B^1_{S^\Psi_t})_{t\geq 0}$ be a one-dimensional subordinate Brownian motion, and
$\uptau_{r_\cD}$ be its first exit time from $\sB_{r_\cD}$. The above estimate gives
\begin{equation}\label{C3.1D}
\sup_{x\in\cD} \ex^x[\uptau_\cD]\; \leq \; \ex^0[\uptau_{r_\cD}].
\end{equation}
Since the L\'{e}vy exponent of $(B^1_{S^\Psi_t})_{t\geq 0}$ is given by $\Psi(u^2)$, we obtain from \cite[Rem.~4.8]{Schilling-98} and \eqref{Psi-mon} that
$$
\ex^0[\uptau_{r_\cD}]\leq \frac{C_2}{\Psi(r_\cD^{-2})}\,,
$$
for some universal constant $C_2$. Hence using \eqref{C3.1D} and the above estimate we obtain the right hand side of
\eqref{C3.1C}.

Finally, consider (iii). In view of \eqref{C3.1B} we only need to show the right hand side of \eqref{C3.1C0}. Using \eqref{C3.1D}
and the estimate above, we get that
\begin{equation}\label{C3.1E}
\sup_{x\in\cD} \ex^x[\uptau_\cD]\leq \frac{C_2}{\Psi(r_\cD^{-2})}= \frac{C_2\lambda_1^\cD}{\Psi(r_\cD^{-2})}
\frac{1}{\lambda_1^\cD}\,.
\end{equation}
On the other hand, using \cite{Chen-Song} and the domain monotonicity of the principal eigenvalue, we obtain
$$
\lambda_1^\cD \leq \Psi(\lambda_{1, \mathrm{Lap}}^\cD)\leq \Psi\left(\frac{\kappa_1}{r^2_\cD}\right)\,,
$$
where $\kappa_1=\lambda_{1, \mathrm{Lap}}^\sB$. A combination with \eqref{C3.1E} gives
\begin{equation}\label{C3.1F}
\sup_{x\in\cD} \ex^x[\uptau_\cD]\leq C_2 \sup_{s\in(0, \infty) }\frac{\Psi(\kappa_1 s)}{\Psi( s)}\frac{1}{\lambda_1^\cD}\,.
\end{equation}
Hence using \eqref{C3.1F} and \eqref{Psi-mon} we find
$$
\sup_{x\in\cD} \ex^x[\uptau_\cD]\leq C_2 \left(1\vee \kappa_1\right)\frac{1}{\lambda_1^\cD}\,.
$$
This completes the proof of \eqref{C3.1C0}.
\end{proof}

\begin{remark}
The bound in \eqref{C3.1B} implies for $p=1$ the well-known relation
$$
\lambda_1^\cD \geq \frac{1}{\sup_{x\in \cD} \ex^{x}[\uptau_{\cD}]},
$$
between the principal Dirichlet eigenvalue and the mean survival time in the domain, first obtained by Donsker and Varadhan
for diffusion processes \cite[eq. (1.2)]{DV-76}. For the classical case $\Psi(u) = u$, the lower bound
$$
{\sup_{x\in \cD} \ex^{x}[\uptau_{\cD}]}\lambda_1^\cD \geq {1}
$$
is known to be sharp for any any bounded domain $\cD$  in $\Rd$ \cite{HLP}.
The moment estimates
$$
\sup_{x\in\cD}\ex^x [\uptau^p_\cD] \geq \frac{\Gamma(p+1)}{(\lambda_1^\cD)^p}, \quad p \geq 1,
$$
have also an independent interest, giving bounds on the (integer and fractional) moments of the mean exit time from $\cD$ for
subordinate Brownian motion, which were not known before. Also, using the same \eqref{EC3.2A}, it follows that $\uptau_\cD$
has $p$-exponential moments of order $p < \lambda_1^\cD$, and we have the bound
$$
\sup_{x\in\cD}\ex^x [e^{p\,\uptau_\cD}] \geq \frac{p}{\lambda_1^\cD-p}, \quad p < \lambda_1^\cD.
$$
\end{remark}

\begin{remark}
There is much important work on estimates similar to \eqref{C3.1C}-\eqref{C3.1C0} for cases when the reference domain is
a simply connected set in $\RR^2$ and the stochastic process is Brownian motion. Much effort has been made on finding the
best possible universal constants for these estimates; see, for instance, \cite{Ban-Unpub, BC-94} and the references
therein. For similar estimates for symmetric stable processes we refer to \cite{BLM-01, Mendez}. Corollary ~\ref{C3.2}
extends the earlier results to subordinate Brownian motion, possibly with non-optimal constants.
\end{remark}

\begin{remark}[Hot-spots]
\label{hotspots}
In the literature the location where the solution of the heat equation in a bounded domain at a given time attains its  maximum
is referred to as a \emph{hot-spot}. Identifying possible hot spots in a convex domain is known to be quite challenging and there
is an extensive literature in this direction. In the case of Neumann boundary conditions the solution approaches the second
eigenfunction on the long run, and the so called Rauch-conjecture states that this eigenfunction attains its maximum on the boundary
of the domain, thus the hot-spots in this case are expected to be located on the edge. This conjecture turned out to be more involved
and false in general, but it has been proven to hold under specific assumptions on the domain, see \cite{BaBu99,BW99, JN00} and
references therein. For Dirichlet boundary conditions the situation is different as now the solution of the heat equation tends to
principal eigenfunction as time goes to infinity, and the hot-spot becomes its maximizer, away from the boundary. In
\cite[Th. ~2.8]{BMS-11} it is shown that there exists a constant $c$, dependent only on $d$, such that for any bounded convex set
$\cD$ one has
\begin{equation}\label{ER3.4A}
\dist(x^*, \partial \cD)\;\geq \; c\, \inrad\cD \Bigl(\frac{\inrad\cD}{\diam\cD}\Bigr)^{d^2-1}\,,
\end{equation}
where $x^*$ denotes a hot-spot of the Laplacian in $\cD$ with Dirichlet boundary condition. Note that
Theorem~\ref{T3.1} improves this result substantially. We single this out in the following result.
\end{remark}
\begin{corollary}[\textbf{Hot-spots}]
Let $\Psi(u)=u$, and $\lambda_{1}^\cD$ be the principal Dirichlet eigenvalue of the
Laplacian for the domain $\cD$, and $\sB$ denote the unit ball centered in the origin. Then
$$
\dist(x^*, \partial \cD)\;\geq \; \sqrt{\frac{\theta}{\lambda_{1}^\sB}}\, \inrad\cD \,,
$$
\label{HScorr}
\end{corollary}
\begin{proof}
By the domain monotonicity property we have
$$
\lambda_1^\cD \;\leq\; \frac{1}{(\inrad\cD)^2}\, \lambda_{1}^\sB.
$$
Using \eqref{inv}, the result follows, possibly with a non-optimal constant.
\end{proof}

\begin{remark}[Universal upper bound on the distance of maximizer]
It is not difficult to see that a reverse inequality to \eqref{ET3.1A} does not hold. Consider the domain
$\cD=[0, \pi]^2$ and the Laplace operator. Then $\varphi_n(x, y) = \sin((2n+1)x)\sin(y)$, $n \in \mathbb N$,
is an eigenfunction with eigenvalue $\lambda_n=(2n+1)^2 + 1$. Note that $\abs{\varphi_n(\frac{\pi}{2},
\frac{\pi}{2})}=1$ and $\dist\left(\left(\frac{\pi}{2}, \frac{\pi}{2}\right), \partial\cD\right)=\frac{\pi}{2}$.
Thus there is no $c > 0$ such that
$$
\frac{\pi}{2}=\dist\left(\left(\frac{\pi}{2}, \frac{\pi}{2}\right), \partial \cD\right)\leq \;
\left(\frac{c}{\lambda_n}\right)^{\nicefrac{1}{2}}\,, \quad \text{for all}\; n \in \mathbb N\,.
$$
We note that $z_n=\big(\frac{\pi}{(2n+1)2}, \frac{\pi}{2}\big)$ is also a maximizer of $\varphi_n$ and
$\dist(z_n, \partial\cD)\sim \lambda_n^{-\nicefrac{1}{2}}$. Therefore an interesting open question is
whether there exists a universal constant $c$ such that the $c\lambda_n^{-\nicefrac{1}{2}}$-neighbourhood
of $\partial\cD$ contains an extremum of the Dirichlet eigenfunction $\varphi_n$.
\end{remark}

\medskip
Inequality \eqref{inv} has another important consequence, which we single out next. Assume that $\Psi$ is
strictly increasing in $(0, \infty)$, denote the Lebesgue measure of $\cD$ by $|\cD|$, and $|\sB|=\omega_d$.
\begin{corollary}[\textbf{Faber-Krahn inequality}]
\label{FKineq}
Under the conditions of Theorem \ref{T3.1} we have
\begin{equation}
\label{FKinequ}
\abs{\cD} \left(\Psi^{-1}\left(\frac{\norm{V^-}_\infty- \inf_{\cD} V^+ +\lambda^{V, \cD}}{\theta}\right)\right)^{\nicefrac{d}{2}}\geq \omega_d.
\end{equation}
\end{corollary}
\begin{proof}
Since $\sB_r(x^*)\subset\cD$ whenever $r = \dist(x^*,\partial\cD)$, using \eqref{inv} it is immediate that
$$
\abs{\cD}\geq |\sB_r(x^*)| \geq
\omega_d\left(\Psi^{-1}\left(\frac{\norm{V^-}_\infty- \inf_{\cD} V^+ +\lambda^{V, \cD}}{\theta}\right)\right)^{-\nicefrac{d}{2}}.
$$
\end{proof}
\noindent
The Faber-Krahn inequality has been previously known only for the classical case $\Psi(u)=u$ \cite[Th. 1.1]{CH12}, and for the
fractional case $\Psi(u)=u^{\nicefrac{1}{2}}$.

\begin{remark}[Torsion]
Recall the notation $H_0 = \Psi(-\Delta)$, and consider the non-local Dirichlet problem
\begin{equation*}
\left\{\begin{array}{ll}
H^{\cD}_0 v = 1, \quad \text{in}\, \; \cD \\
\quad \;\; v = 0,\, \quad \text{in}\,\; \cD^c.
\end{array} \right.
\end{equation*}
The function $v$ is called torsion, and recently it has been noticed that its maximizer and the maximizer $x^*$ of
the principal Dirichlet eigenfunction of $H_0$ are located very near to each other, though do not coincide. This
puzzling phenomenon has been discussed in \cite{BLMS}, see also the references therein. Note that the solution has
the immediate probabilistic meaning $v(x)=\ex^x[\uptau_\cD]$. It is immediate from Corollaries~\ref{C3.1}-\ref{C3.2}
above that with a constant $C = C(d) > 0$, we have
$$
\sup_{\cD} v(x)\leq C v(x^*),
$$
A similar result was obtained in \cite[Cor.~2]{RS-17} for the case of the classical Laplacian in dimension $2$.
Moreover, for $\Psi(u)=u$, an estimate similar to \eqref{C3.1C0} is also known \cite{VDB}.
\label{tors}
\end{remark}

\section{Compactly supported potentials}
In this section we consider the eigenvalue problems \eqref{dirichev}-\eqref{fullev} for the special choice of
bounded potentials with compact support. In case $V = -v\Ind_\cK$ with a bounded set $\cK \subset \Rd$ with
non-empty interior, we say that $V$ is a potential well with coupling constant $v > 0$.

Concerning the eigenvalue problem in $L^2(\Rd)$, recall that the non-local Schr\"odinger operator $H = \Psi(-\Delta)
+V$ admits a Feynman-Kac representation \cite{HIL12} of an eigenfunction $\varphi$ in the form
$$
e^{-tH}\varphi(x) = e^{\lambda t}\ex^x[e^{-\int_0^t V(X_s)ds}\varphi(X_t)],  \quad x \in \Rd, \, t \geq 0.
$$
For a potential well $- v \Ind_\cK$ this becomes specifically
$$
e^{-tH}\varphi(x) = e^{\lambda t}\ex^x[e^{vU_t^\cK(X)}\varphi(X_t)],
$$
where
$$
U_t^\cK(X) = \int_0^t \Ind_\cK(X_s) ds
$$
is the occupation measure of the set $\cK$ by subordinate Brownian motion $\pro X$.

For non-local Schr\"odinger operators $H$ above the semigroup $\semi T$, $T_t = e^{-tH}$, is well-defined and
strongly continuous. For all $t>0$, every $T_t$ is a bounded operator on every $L^p(\Rd)$ space, $1 \leq p \leq
\infty$. The operators $T_t: L^p(\Rd) \to L^p(\Rd)$ for $1 \leq p \leq \infty$, $t > 0$, and $T_t: L^p(\Rd) \to
L^{\infty}(\Rd)$ for $1 < p \leq \infty$, $t \geq t_{\rm b}$, and $T_t: L^1(\Rd) \to L^{\infty}(\Rd)$ for $t
\geq 2t_{\rm b}$ are bounded, with some $t_{\rm b} \geq 0$. Also, for all $t \geq 2t_{\rm b}$, $T_t$ has a bounded
measurable kernel $u(t, x, y)$ symmetric in $x$ and $y$, i.e., $T_t f(x) = \int_{\Rd} u(t,x,y)f(y)dy$, for all $f
\in L^p(\Rd)$ and $1 \leq p \leq \infty$. For all $t>0$ and $f \in L^{\infty}(\Rd)$, $T_t f$ is a bounded continuous
function. Thus the eigenfunctions solving \eqref{fullev} are bounded and continuous, whenever they exist. Also, they
have a pointwise decay to zero at infinity. For a subclass of subordinate Brownian motions it is known that the
eigenfunctions decay at a rate determined by the L\'evy density of $\pro X$. For further details we refer to \cite{KL17}.

Since these potentials are relatively compact perturbations of $H_0=\Psi(-\Delta)$, the essential spectrum is preserved,
and thus we have $\Spec H = \Spec_{\rm ess} H \cup \Spec_{\rm d}H$, with $\Spec_{\rm ess} H = \Spec_{\rm ess} H_0 =
[0 ,\infty)$. The existence of a discrete component depends on further details of the operator. Generally,
$\Spec_{\rm d} H \subset (-v, 0)$, and $\Spec_{\rm d} H$ consists of at most a countable set of isolated eigenvalues
of finite multiplicity whenever it is non-empty, with possible accumulation point up to zero. For non-negative compactly
supported potentials it is known that $\Spec_{\rm d} H \neq \emptyset$ if $\pro X$ is a recurrent process \cite{CMS90}. For
potential wells this means that at least an eigenfunction exists for every $v > 0$ when $\pro X$ recurrent, on the other
hand, it is also possible to show that for transient processes eigenfunctions do not exist if $v$ is too small, but there
is at least one if $v$ is large enough.


For the remainder of this section we assume that an eigenfunction exists in either case \eqref{dirichev}-\eqref{fullev},
which is thus bounded and continuous. The following result applies for both eigenvalue problems.

\begin{theorem}\label{T4.1}
Let $\cD$ be $\Rd$ or a bounded subset of $\Rd$, $V$ be a convex increasing function attaining a global minimum
at $\hat{x}\in\cD$, and consider the respective eigenvalue equations with a pair $\lambda$ and $\varphi$ such that
$\lim_{x\to z\in\partial \cD} \varphi(x)=0$. Furthermore, let $x^*$ be the location of a global maximum of $\abs{\varphi}$,
and consider the set
$$
\sU_\lambda = \{x \in \cD: \, V(x) \leq \lambda\}\cap\cD.
$$
Then we have $x^*\in \sU_\lambda$ and hence,
$$
\dist(x^*, \hat{x})\leq \max_{z\in\partial \sU_\lambda}\dist(x^*, z).
$$
\end{theorem}

\begin{proof}
Note that we only need to show that $x^*\in \sU_\lambda$. Assume, to the contrary, that $x^*\in \cD\cap \sU_\lambda^c$.
Clearly, we have $V >\lambda$ on $\cD\cap \sU_\lambda^c$.  Also, we may assume that $\varphi(x^*)>0$. Using the strong
Markov property in the Feynman-Kac representation, we find that
$$
\varphi(x^*)=\ex^{x^*}\left[e^{-\int_0^{t\wedge\uptau_{\sU^c_\lambda}} (V(X_s)-\lambda)\, \D{s}}
 \varphi(X_{t\wedge\uptau_{\sU_\lambda}})\Ind_{\{t\wedge\uptau_{\sU^c_\lambda}<\uptau_\cD\}}\right]\,,
$$
which implies,
$$
1 \leq
\ex^{x^*}\left[e^{-\int_0^{t\wedge\uptau_{\sU_\lambda}} (V(X_s)-\lambda)\,\D{s}}
\Ind_{\{t\wedge\uptau_{\sU^c_\lambda}<\uptau_\cD\}}\right]\,, \quad \forall \; t>0\,.$$
However, the above is not possible and hence this is a contradiction.
\end{proof}
\noindent
We note that in \cite{B17} a related question has been addressed for classical Dirichlet-Schr\"odinger equations in convex
planar domains.

\begin{remark}\label{R4.1}
It should noted that the convexity of $V$ is not used to find the location of the maximizer. For instance, if we have $V$
compactly supported inside $\cD$ and $\lambda < 0$, then the same proof above shows that $x^*\in\supp V$.
\end{remark}

Next we consider a situation in a bounded domain. The following result shows how far from the support of a not sufficiently
negative potential inside a bounded domain a maximizer can move out. This may be compared with Theorem~\ref{T3.1}, in
particular, it will be seen that the effect of the potential is exercised by the eigenvalue alone. We will use the following
condition on $\Psi$ repeatedly, which we single out here.
\begin{assumption}
\label{HT4.1}
Let $\Psi \in \mathcal B_0$. We assume that for every $\gamma_0 > 0$
\begin{equation}
\lim_{s\to 0} \, \sup_{\gamma\in[\gamma_0, \infty)}\, \frac{\Psi(s\gamma)}{\Psi(\gamma)}=0\, 
\end{equation}
holds.
\end{assumption}
\noindent
We will comment and give some examples of Bernstein functions satisfying this assumption following the proof of our
next main result.

To explain our next result, consider a potential $V$ compactly supported in $\cD$. Note that the lower bound in
\eqref{inv} uses the $L^\infty$ norm of $V$ and therefore it is difficult to say how the size of the support
of $V$ influences the location $x^*$. In the next result we make an attempt in this direction. In particular,
we show that if $x^*$ stays for some reason sufficiently far from $\supp V$, then the lower bound in \eqref{inv}
improves. While the assumption may not be easily verifiable at this stage, we find the conclusion interesting as
it highlights a mechanism of the delicate balance phenomenon driving the maximizer $x^*$ to stabilise.

\begin{theorem}\label{T3.2}
Let $\Psi$ satisfy Assumption \ref{HT4.1}, and $V$ be a potential with compact support $\supp V = \cK$. Consider a convex
bounded domain $\cD \subset \Rd$, containing $\cK$, and let $\dist(\cK, \partial\cD)=\kappa>0$. Also, let $\varphi$ be an
eigenfunction at eigenvalue $\lambda>0$ solving \eqref{dirichev}, and suppose it is known about a global maximizer $x^*$
of $\abs{\varphi}$ that $\dist(x^*, \partial\cD) \leq\kappa/2$. Then there exists a constant $\zeta > 0$, dependent on $d$,
$\kappa$ and $\Psi$, but not on $\cD$, $\cK$, $\varphi$ or $\lambda$, such that
\begin{equation}\label{ET4.1A}
\dist(x^*, \partial\cD) \geq \frac{1}{\sqrt{\Psi^{-1}\left(\frac{\lambda}{\zeta}\right)}}.
\end{equation}
\end{theorem}
\begin{proof}
Denote $r=\dist(x^*, \partial\cD)$, and
without loss of generality assume that $x^*=0$.
Let $t=\frac{c}{\Psi(r^{-2})}$, where the constant $c$ will be chosen below. From the proof of Theorem~\ref{T3.1} it
follows that we can choose $c$ large enough to satisfy
\begin{equation}\label{ET4.1B}
\Prob_{S}(S^\Psi_{t}< r^2)<\frac{1}{4}, \quad \forall\; r>0\,.
\end{equation}
Fix this choice of $c$ and define $T_c=\frac{c}{\Psi(4\kappa^{-2})}$.  Since $r \leq \frac{\kappa}{2}$, we have $t\leq T_c$.
Using \eqref{HT4.1} we show below that there exists $T_\circ>1$ such that
$$
\Prob_S(S^\Psi_t\leq  r^2 T_\circ)\geq \frac{1}{2}, \quad \text{for all}\;\, 0<\;r<\; \frac{\kappa}{2}\,.
$$
Define $Y_r=\frac{1}{r^2}S^\Psi_t$. Then the Laplace transform of $Y_r$ is given by
\begin{equation}\label{ET4.1C}
\Breve{f}(s)= \ex\left[e^{-sY_r}\right]= \ex\left[e^{-\frac{s}{r^2}S^\Psi_t}\right]=e^{-\frac{c}{\Psi(r^{-2})}\Psi(sr^{-2})}\,.
\end{equation}
Since $r<\kappa/2$, using \eqref{HT4.1} and \eqref{ET4.1C} we see that $\Breve{f}(s)\to 1$ as $s\to 0$, uniformly in
$r\in(0, \frac{\kappa}{2}]$. Thus by the uniform Tauberian theorem \cite[Th.~3]{Lai}, we obtain
$$
\Prob_S(Y_r\leq y)\to 1 \quad \text{as}\; y\to\infty, \quad \text{uniformly in}\; r\in(0, \frac{\kappa}{2}].
$$
Hence we can find $T_\circ>1$ satisfying
\begin{equation}\label{ET4.1D}
\Prob_S(S^\Psi_t\leq r^2 T_\circ)\geq \frac{1}{2}, \quad \text{for all}\;\, 0<r<\frac{\kappa}{2}\,.
\end{equation}
Combining \eqref{ET4.1B} and \eqref{ET4.1D} we obtain that
\begin{equation}\label{ET4.1E}
\Prob_{S}\left(S^\Psi_{t}\in [r^2, r^2 T_\circ]\right) > \frac{1}{4}, \quad \forall\; r\in(0, \frac{\kappa}{2}]\,.
\end{equation}
Now we fix the above choice of $T_\circ$, which depends on $c, \kappa$ and $\Psi$. On the other hand, since $\cD$ is convex, we
may assume that there exists a point $z_0\in\partial\cD$ such that $\dist(0, z_0)=r$, $z_0$ lies on the $x_1$-axis and
$\cD$ lies of the on the complement of the half-space $\{y\in\Rd :\, z_0\cdot y \geq r^2\} $. Define $\chi:[0, T_\circ]\to \Rd$
by
$$
\chi(s)=2\sqrt{s}e_1,
$$
where $e_1$ is the unit vector along the $x_1$-axis. Note that $\dist(z_0, \chi(r^2))=r$ and $z_0\cdot\chi(r^2)=2r^2$. Define
for $\delta\in (0, \frac{\kappa}{4}\wedge\frac{1}{2})$
$$
N_\delta = \left\{f\in \cC([0, T_\circ], \Rd)\; :\; f(0)=0 \;\, \text{and} \; \max_{s\in[0, T_\circ]}\abs{f(s)-\chi(s)}<\delta\right\}\,,
$$
i.e., a $\delta$-neighbourhood of $\chi$ in $\cC_0([0, T_\circ],\Rd)$, the space of $\Rd$-valued continuous functions on
$[0, T_\circ]$ with value $0$ at $s=0$. By the Stroock-Varadhan support theorem it follows that there exists $\delta_1 > 0$
such that
\begin{equation}\label{ET4.1F}
\Prob^0_W\left(\frac{1}{r}B_{r^2s} \in N_\delta\right) = \Prob^0_W(B_s\in N_\delta)=\delta_1>0\,.
\end{equation}
Note the equivalence of the events
\begin{align*}
\left\{\max_{s\in[0, T_\circ]} \Big|\frac{1}{r}B_{r^2s} -\chi(s)\Big| < \delta \right\}
= \left\{\max_{s\in[0, T_\circ]} \abs{B_{r^2s} -r \chi(s)}< r\delta \right\}
= \left\{\max_{s\in[0, T_\circ r^2]} \abs{B_{s} - \chi(s)}< r\delta \right\},
\end{align*}
where in the last equality we used that $r\chi(s)=\chi(r^2s)$. Thus we find
\begin{equation}\label{ET4.1G}
\Prob^0_W\left(\max_{s\in[0,  r^2 T_\circ]} \abs{B_{s} - \chi(s)}< r\delta\right)=\delta_1\,.
\end{equation}
Combining \eqref{ET4.1E} and \eqref{ET4.1G} we have
\begin{align}\label{ET4.1H}
&
\Prob^{0}\left((\omega, \varpi) : \sup_{s\in[0, t]}\abs{B_{S^\Psi_s}-\chi(S^\Psi_s)}< r\delta, \,
S^\Psi_t\in [r^2, r^2 T_\circ]\right)
\\
&\qquad
\geq \Prob^{0}\left((\omega, \varpi) : \max_{s\in[0, r^2 T_\circ]} \abs{B_{s}(\omega) - \chi(s)}<
r\delta, S^\Psi_t(\varpi)\in [r^2, r^2 T_\circ]\right)\nonumber
\\
&\qquad
= \Prob^0_W\left(\max_{s\in[0, r^2 T_\circ]} \abs{B_{s} - \chi(s)}< r\delta\right)\,
\Prob_S(S^\Psi_t\in [r^2, r^2 T_\circ]) \; \geq \; \frac{\delta_1}{4}
\nonumber,
\end{align}
where the third line follows from the independence of Brownian motion and the subordinator. By the construction of
$\chi$ it is seen that every path satisfying
$$
\sup_{s\in[0, t]}\abs{B_{S^\Psi_s}-\chi(S^\Psi_s)}< r\delta, \;\; S^\Psi_t\in [r^2, r^2 T_\circ],
$$
must leave $\cD$ by time $t$ since $B_{S^\Psi_t}\in\cD^c$, and it does not enter $\cK$ in the time interval $[0,t]$.
Thus by \eqref{ET4.1H} we obtain
\begin{equation}\label{ET4.1I}
\Prob^0(\uptau_\cD\leq  t\wedge \uptau_{\cK^c})>\frac{\delta_1}{4}\,.
\end{equation}
Then by the Feynman-Kac formula and the strong Markov property it follows that
$$
\varphi(0)=\ex^0\left[e^{\lambda (t\wedge\uptau_{\cK^c})} \varphi(X_{t\wedge\uptau_\cK})\Ind_{\{t\wedge\uptau_{\cK^c}<\uptau_\cD\}}\right]
\leq \varphi(0) e^{\lambda t} \Prob^0(t\wedge\uptau_{\cK^c}<\uptau_\cD)\leq \varphi(0) e^{\lambda t} (1-\frac{\delta_1}{4})\,,
$$
using \eqref{ET4.1I}. By taking logarithms both sides, we obtain \eqref{ET4.1A}.
\end{proof}

There is a large family of subordinate Brownian motions satisfying Assumption \ref{HT4.1}. First we show a general
statement and then illustrate it by some important examples.
\begin{lemma}\label{L3.5}
Suppose that $\Psi$ is unbounded and regularly varying at infinity, i.e., with a slowly varying function $\ell$
and constant $\beta>0$ we have
$$
\Psi(u) \asymp u^\beta \ell(\beta), \quad \text{for all large} \; u.
$$
Then Assumption \ref{HT4.1} holds.
\end{lemma}
\begin{proof}
It suffices to show that for any sequence $(s_n, \gamma_n)$ with
$$
s_n\to 0, \quad \gamma_n\to \infty, \quad \text{and}\, \; s_n\gamma_n\to \infty,
$$
we have
\begin{equation}
\label{goto}
\lim_{n\to\infty} \frac{\Psi(s_n\gamma_n)}{\Psi(\gamma_n)}= 0.
\end{equation}
Fix any $\varepsilon>0$. Then for large $n$,
$$
\frac{\Psi(s_n\gamma_n)}{\Psi(\gamma_n)}\leq \frac{\Psi(\varepsilon\gamma_n)}{\Psi(\gamma_n)}\asymp
\varepsilon^\beta \, \frac{\ell(\varepsilon\gamma_n)}{\ell(\gamma_n)}\asymp \varepsilon^\beta.
$$
Hence \eqref{goto} follows.
\end{proof}

\begin{example}
By Lemma~\ref{L3.5} the following Bernstein functions satisfy Assumption \ref{HT4.1}:
\begin{itemize}
\item[(i)] $\Psi(u)=u^{\alpha/2}, \, \alpha\in(0, 2]$.
\item[(ii)] $\Psi(u)=(u+m^{2/\alpha})^{\alpha/2}-m, \, m> 0, \alpha\in (0, 2)$.
\item[(iii)]$\Psi(u)=u^{\alpha/2} + u^{\beta/2}, \, 0 < \beta < \alpha \in(0, 2]$.
\item[(iv)] $\Psi(u)=u^{\alpha/2}(\log(1+u))^{\beta/2}$, $\alpha \in (0,2)$, $\beta \in (0, 2-\alpha)$.
\item[(v)] $\Psi(u)=u^{\alpha/2}(\log(1+u))^{-\beta/2}$, $\alpha \in (0,2]$, $\beta \in [0,\alpha)$.
\end{itemize}
\end{example}

\begin{example}
On the other hand,  $\Psi(u)=\log(1+u^{\alpha/2})$, $\alpha\in (0,2]$, does not satisfy Assumption
\ref{HT4.1}. To see this note that for $s=\frac{1}{n}$ and $\gamma=n^2$ we have
$$
\lim_{n\to\infty} \frac{\Psi(s\gamma)}{\Psi(\gamma)} = \lim_{n\to\infty}
\frac{\log(1+ n^{\alpha/2})}{\log(1+n^\alpha)}\geq \frac{1}{2}\,.
$$
\end{example}

In the remaining part of this section we consider the eigenvalue problem in full space.
\begin{theorem}\label{T4.3}
Consider the operator $H$ given by \eqref{nonlocSch}, $\supp V = \cK$, and let $\varphi$ be a solution of the
Schr\"odinger eigenvalue problem \eqref{fullev} for $H$, corresponding to eigenvalue $\lambda = - |\lambda|<0$.
If $|\varphi|$ has a global maximum at $x^* \in \Rd$, then $x^* \in \cK$.
\end{theorem}

\begin{proof}
We show that there is no maximizer in $\cK^c$. Assume, to the contrary, that $x^*\in \cK$. Therefore, for a suitable
$\delta>0$ we have $\sB_\delta(x^*)\in \cK^c$. Let $\uptau_\delta$ be the exit time from the ball $\sB_\delta(x^*)$.
Since $\ex^{x^*}[\uptau_\delta]>0$, we find $t>0$ such that $\Prob^{x^*}(\uptau_\delta> t)>0$. As before, we can also
assume that $\varphi(x^*)>0$. By the Feynman-Kac representation we have
\begin{align*}
\varphi(x^*)
&=
\ex^{x^*}\left[e^{\lambda(t\wedge\uptau_\delta)}\varphi(X_{t\wedge\uptau_\delta})\right]
\\
&\leq
e^{\lambda t}\ex^{x^*}\left[\varphi(X_{t})\Ind_{\{\uptau_\delta>t \}}\right] +
\left[e^{\lambda\uptau_\delta}\varphi(X_{t})\Ind_{\{\uptau_\delta \leq t \}}\right]
\\
&\leq e^{\lambda t} \varphi(x^*) \Prob^{x^*}(\uptau_\delta>t) + \varphi(x^*) \Prob^{x^*}(\uptau_\delta\leq t).
\end{align*}
This would imply $e^{\lambda t}>1$, which is a contradiction as $\lambda<0$. Hence $x^*\in\cK$.
\end{proof}

\begin{remark}
Recall that the eigenfunctions are continuous, as mentioned earlier. Since $V$ is bounded, from the Feynman-Kac
representation we have for every $t>0$ that
\begin{equation}\label{ER4.2A}
\abs{\varphi(x)}\;\leq\; e^{(\norm{V}_\infty -\lambda)t} \Exp^x[|\varphi|^2(X_t)]^{\nicefrac{1}{2}}
\leq e^{(\norm{V}_\infty -\lambda_0)t} \Bigl(\int_{\Rd} |\varphi(y)|^2 q_t(x-y)\, \D{y}\Bigr)^{\nicefrac{1}{2}}\,,
\end{equation}
where $q_t(x, y) = q_t(x-y)$ denotes the transition density of $\pro X$ starting at $X_0=x$. It follows by
subordination, see \eqref{subord}, that
$$
q_t(x-y) = \int_0^\infty \frac{1}{(4\pi s)^{\nicefrac{d}{2}}} e^{-\frac{|x-y|^2}{4s}} \Prob_{S}(S^\Psi_t\in \D{s}).
$$
Therefore, for every fixed $y$ we have $q_t(x-y)\to 0$ as $\abs{x}\to\infty$. Moreover, if $\Psi$ satisfies the
Hartman-Wintner condition \eqref{HW}, then $q_t(x, y)$ is bounded and continuous. Hence by dominated convergence
we obtain from \eqref{ER4.2A} that $\lim_{\abs{x}\to \infty} \abs{\varphi(x)}=0$, thus every eigenfunction attains
its maximum in $\Rd$.
\end{remark}

Finally, we show how deep inside the support the maximizer can be for a potential well. We denote by $\Int \cK$ the
interior of $\cK$.
\begin{theorem}\label{T3.3}
Let $V=-v\Ind_\cK$ with a bounded convex set $\cK$, and $\varphi$ be an eigenfunction corresponding to eigenvalue $\lambda
= -|\lambda|<0$ solving the eigenvalue problem \eqref{fullev}. Suppose that $\Psi$ is unbounded and satisfies Assumption
\ref{HT4.1}. Then there exist two constants $\varrho_1, \varrho_2 > 0$, dependent only on $\Psi$ and $\inrad \cK$, such
that if
\begin{equation}\label{ET3.3A}
\frac{v-\abs{\lambda}}{\abs{\lambda}}\;\leq\;\varrho_1\,,
\end{equation}
then $x^*\in \Int \cK$ and
\begin{equation}\label{ET3.3B}
\dist(x^*, \partial\cK)\geq \frac{1}{\sqrt{\Psi^{-1}\left(\frac{\abs{\lambda}}{\varrho_2}\right)}}\, .
\end{equation}
\end{theorem}

\begin{proof}
\emph{Step 1:} First we prove \eqref{ET3.3B} assuming that $x^*\in \Int \cK$. By a shift we can assume that $x^*=0$ with
no loss of generality, and we denote $r=\dist(x^*, \partial\cK)>0$. Let $t=\frac{c}{\Psi(r^{-2})}$ where the constant $c$
will be chosen later. From the proof of Theorem~\ref{T3.2} we see that we can choose $c$ large enough such that
\begin{equation}\label{ET3.3C}
\Prob_{S}(S^\Psi_{t}\in [r^2, r^2 T_\circ])=\delta_1>\frac{1}{4}, \quad \forall\; r\in(0, \inrad \cK]\,;
\end{equation}
see \eqref{ET4.1E} above. Therefore, by the independence of increments we have from \eqref{ET3.3C} that
\begin{equation}\label{ET3.3D}
\Prob_{S}(S^\Psi_{t}\in [r^2, r^2 T_\circ], \, S^\Psi_{2t}-S^\Psi_t\in [r^2, r^2 T_\circ])=\delta_1^2,
\quad \forall\; 0<r\leq\inrad \cK\,.
\end{equation}
Now we fix the above choice of $T_\circ$ which depends on $c$ and $\Psi$ and $\inrad \cK$ (recall that $r$ and $t$ are related).
Since $\cK$ is convex, we may assume that the point $z_0 =
(r, 0, \ldots, 0)\in\partial\cK$ is such that $\dist(0, z_0)=r$, $\cK$ lies on the on the complement of the half-space
$\{y\in\Rd\;:\; z_0\cdot y \geq r^2\} $. Define $\chi:[0, T_\circ]\to \Rd$ by
$$\chi(s)=2 s,$$
Note that $\dist(0, \chi(\frac{1}{2}r))=r$. Define
$$
\cN= \left\{f\in \cC([0, 2T_\circ], \RR)\; :\; f(0)=0 \; \text{and} \; \max_{s\in[0, 2T_\circ]}\abs{f(s)-\chi(s)}<\frac{1}{2}\right\}.
$$
In a similar manner as in the proof of Theorem \ref{T3.2} we find that there is a $\delta_2 > 0$ such that
\begin{equation}\label{ET3.3E}
\Prob^0_W\left(\frac{1}{r}B_{r^2\cdot} \in \cN\right) = \Prob^0_W(B^1_\cdot \in \cN)=\delta_2.
\end{equation}
Also, we have
\begin{align*}
&\left\{\max_{s\in[0, 2T_\circ]} \Big|\frac{1}{r}B^1_{r^2s} -\chi(s)\Big|<\frac{1}{2}\right\}
= \left\{\max_{s\in[0, 2 r^2 T_\circ]} \Big|B^1_{s} - \frac{1}{r}\chi(s)\Big|< \frac{r}{2} \right\},
\end{align*}
using scaling and that $r\chi(s)=\frac{1}{r}\chi(r^2s)$. Thus we obtain
\begin{equation}\label{ET3.3G}
\Prob^0_W\left(\max_{s\in[0, 2 r^2 T_\circ]} \abs{B^1_{s} - \frac{1}{r}\chi(s)}< \frac{r}{2}\right)=\delta_2\,.
\end{equation}
Combining \eqref{ET3.3D} and \eqref{ET3.3G} gives
\begin{align}\label{ET3.3H}
&\Prob^{0}\left((\omega, \varpi) : \max_{s\in[0, 2T_\circ r^2]}
\abs{B^1_{s}(\omega) - \frac{1}{r}\chi(s)}< r/2, \, S^\Psi_t(\varpi)\in [r^2, r^2 T_\circ],
S^\Psi_{2t}(\varpi)-S^\Psi_t(\varpi)\in [r^2, r^2 T_\circ]\right)\nonumber \\
&=
\Prob^0_W\left(\max_{s\in[0, 2T_\circ r^2]} \abs{B^1_{s} - \chi(s)}< r/2\right)\,
\Prob_S(S^\Psi_t\in [r^2, r^2 T_\circ], \, S^\Psi_{2t}-S^\Psi_t\in [r^2,r^2 T_\circ]) =\delta_2\, \delta_1^2,
\end{align}
where the third line follows from the independence of the two processes. Let
$$
\widehat\Omega = \left\{(\omega, \varpi) : \; \max_{s\in[0, 2T_\circ r^2]} \Big|B^1_{s}(\omega) - \frac{1}{r}\chi(s)\Big|
< \frac{r}{2}, \;\;  S^\Psi_t(\varpi)\in [r^2, r^2 T_\circ],  \;\; S^\Psi_{2t}(\varpi)-S^\Psi_t(\varpi)\in [r^2, r^2 T_\circ]\right \}\,.
$$
We see that
$$
\widehat\Omega\subset\left\{(\omega, \varpi) : \;
\sup_{s\in[0, t]}\Big|B^1_{S^\Psi_s}-\frac{1}{r}\chi(S^\Psi_s) \Big|< \frac{r}{2}, \; \;  S^\Psi_t\in [r^2, r^2 T_\circ], \;\;
S^\Psi_{2t}-S^\Psi_t\in [r^2, r^2 T_\circ]\right\}\,.
$$
By the construction of $\chi$ it follows that for every $(\omega, \varpi)\in\widehat\Omega$, $B_{S^\Psi_t}\in \cK^c$ and the paths
of $B_{S^\Psi_s}$ stay in $\cK^c$ for all $s\in [t, 2t]$. This observation will play a key role in our analysis below.

Let $\delta=\delta_2\,\delta_1^2$, and define
$$
2\varrho_1=\frac{\delta}{2-\delta}\in (0, 1),
$$
and a function $\xi:\RR\to\RR^+$ by
$$
\xi(y)=\delta e^{-\frac{1}{2}(1-\varrho_1) y} + (1-\delta)e^{\varrho_1 y}.
$$
It is direct to see that $\xi'(\varepsilon_0)=0$ gives
$$
\varepsilon_0=\frac{2}{1+\varrho_1}\log\frac{\delta(1-\varrho_1)}{2\varrho_1(1-\delta)}.
$$
Since
$$
\varrho_1 <\frac{\delta}{2-\delta} \quad \mbox{implies} \quad \frac{\delta(1-\varrho_1)}{2\varrho_1(1-\delta)}>1,
$$
we have $\varepsilon_0>0$.  Again observe that $\xi'(0)<0$, and therefore $\xi(y)<1$ for $y\in (0, \varepsilon_0)$.

Suppose now that $\frac{v-\abs{\lambda}}{\abs{\lambda}}\leq\varrho_1$. By the Feyman-Kac representation we have
$$
\varphi(0)=\ex^0\left[e^{\int_0^{2t} (\lambda-V(X_s))\, \D{s}}\varphi(X_{2t})\right]\,,
$$
which, in turn, implies
\begin{align}\label{ET3.3I}
1 &\leq
\ex^0\left[e^{\int_0^{2t} (\lambda-V(X_s))\, \D{s}}\right]\nonumber
\\
&= \ex^0\left[e^{\int_0^{2t} (\lambda-V(X_s))\, \D{s}}\Ind_{\widehat\Omega}\right] +
\ex^0\left[e^{\int_0^{2t} (\lambda-V(X_s))\, \D{s}}\Ind_{\widehat\Omega^c}\right]\nonumber
\\
&\leq
\ex^0\left[e^{\int_0^{2t} (\lambda-V(X_s))\, \D{s} + \lambda t } \Ind_{\widehat\Omega}\right]
+ (1-\delta) e^{(v+\lambda)2t}\nonumber
\\
& \leq \delta e^{(v+\lambda) t - \abs{\lambda} t} + (1-\delta) e^{(v+\lambda)2t}
\leq \delta e^{\varrho_1\abs{\lambda}t - \abs{\lambda} t} + (1-\delta) e^{\varrho_1\abs{\lambda}2t}
=\xi(2t\abs{\lambda})\,,
\end{align}
where in the fourth line we used \eqref{ET3.3H}. Since $2t\abs{\lambda}>0$ and $\xi(2t\abs{\lambda})\geq 1$,
we conclude that
$$
2t\abs{\lambda}\geq \varepsilon_0\,
$$
holds. Hence \eqref{ET3.3B} follows with $\varrho_2=\frac{\varepsilon_0}{2c}$.

\medskip
\noindent
\emph{Step 2:} To conclude, we prove that under the condition \eqref{ET3.3A} we have $x^*\notin\partial\cK$.
Like before, we may assume that $x^*=0$ and  $\cK\subset\{x_1\leq 0\}$. Note that the estimate \eqref{ET3.3H}
holds uniformly in $r\in(0, \inrad \cK)$. Since $0$ is on the boundary of $\cK$ and the function $\chi$, defined
above, lies in $\{x_1\geq 0\}$, we observe that for every $r>0$ and every $(\omega, \varpi)\in \widehat\Omega=
\widehat\Omega_r$ we have $B_{S^\Psi_t}\in \cK^c$ and the paths $B_{S^\Psi_s}$ in stay $\cK^c$ for $s\in [t, 2t]$,
where $t=\frac{c}{\Psi(r^{-2})}$ and $c$ is chosen the same as before. Therefore, following a similar argument as
in the proof of \eqref{ET3.3I}, we obtain
$$
1 \leq \xi(2t\abs{\lambda}),
$$
for all $r>0$. Since $t\to 0$ as $r\to 0$, and since $\Psi$ is unbounded, the above estimate cannot hold for
small $t$. Thus we have a contradiction showing that $0=x^*\in \Int\cK$.
\end{proof}

\begin{remark}
We note that for a potential well $V=-v\Ind_\cK$, $v>0$, we have
$$
v-\abs{\lambda}\leq \lambda_1^\cK,
$$
where $\lambda_1^\cK$ is the principal eigenvalue of $\Psidel$ in $\cK$ with Dirichlet exterior condition on
$\cK^c$. Indeed, from the Feynman-Kac formula we get that
\begin{align*}
\varphi(x) &\geq \Exp^x\left[e^{\int_0^t(v\Ind_{\cK}(X_s)+\lambda)\, \D{s}} \varphi(X_t)\Ind_{\{t<\uptau_\cK\}}\right]
\\
&\geq e^{t (v-\abs{\lambda})}\, \min_{y\in\cK} \varphi(y)\, \Prob^x(t<\uptau_\cK)\,, \quad x\in\cK.
\end{align*}
By taking logarithms on both sides and dividing by $t > 0$, we get
\begin{equation*}
v-\abs{\lambda} \leq -\limsup_{t\to\infty}\, \frac{1}{t}\log\Prob^x(t<\uptau_\cK)\leq \lambda_1^\cK\,.
\end{equation*}
Thus the numerator at the left hand side of \eqref{ET3.3A} is always bounded by $\lambda_1^\cK$, and so for
$\abs{\lambda}$ large enough \eqref{ET3.3A} holds.
Also, notice that the result in Theorem~\ref{T3.3} continues to hold for more general potentials $V$ supported on $\cK$
and $\lambda<0$. In this situation \eqref{ET3.3A} will be replaced by
$$
\frac{-\min_{x\in\cK} V(x)-\abs{\lambda}}{\abs{\lambda}}\;\leq\;\varrho_1\,.
$$
\end{remark}

\bigskip
Notice that the dependence of $\varrho_1$ and $\varrho_2$ on $\inrad \cK$ comes from \eqref{HT4.1}, which has
been crucially used in \eqref{ET4.1D}. This dependence can be waived for a class of $\Psi$ for which \eqref{HT4.1}
holds uniformly in $\gamma_0>0$, i.e., when
 \begin{equation}\label{E4.20}
\lim_{s\to 0} \, \sup_{\gamma\in(0, \infty)}\, \frac{\Psi(s\gamma)}{\Psi(\gamma)}=0\,.
\end{equation}
Observe that if $\Psi$ satisfies Assumption \ref{WLSC}, then \eqref{E4.20} holds. Indeed, we have then
$$
\lim_{s\to 0} \, \sup_{\gamma\in(0, \infty)}\, \frac{\Psi(s\gamma)}{\Psi(\gamma)}
= \lim_{s\to 0} \, \sup_{\gamma\in(0, \infty)}\, \frac{\Psi(s\gamma)}{\Psi(s^{-1}s\gamma)}\lesssim
\lim_{s\to 0} s^\mu =0.
$$
Moreover, \eqref{ET4.1D}-\eqref{ET4.1E} follow then uniformly in $r\in(0, \infty)$. Therefore, in this case
$\varrho_1$ and $\varrho_2$ only depend on $\Psi$ and not on $\inrad \cK$. This is recorded in the following result.

\begin{theorem}\label{T3.4}
Suppose that $\Psi$ satisfies Assumption \ref{WLSC}, and let $\varphi$ and $\lambda = -|\lambda|$ solve the eigenvalue
equation \eqref{fullev} for $H$ with a potential well $V = -v \Ind_\cK$. Then there exist positive $\varrho_1, \varrho_2$,
dependent only on $\Psi$, such that if
\begin{equation*}
\frac{v-\abs{\lambda}}{\abs{\lambda}}\;\leq\;\varrho_1\,,
\end{equation*}
then $x^*\in\Int\cK$ and
\begin{equation*}
\dist(x^*, \partial\cK)\geq \frac{1}{\sqrt{\Psi^{-1}\left(\frac{\abs{\lambda}}{\varrho_2}\right)}}\, .
\end{equation*}
\end{theorem}

\medskip
Theorems \ref{T3.3}-\ref{T3.4} have the following interesting ``no-go" type consequence.
\begin{corollary}
\label{nogo}
Under the conditions of Theorems \ref{T3.3}-\ref{T3.4} we have that whenever
$$
v < \varrho_2 \Psi\left(\frac{1}{(\inrad \cK)^2}\right)\,,
$$
then either $\abs{\lambda} < v/(1+\varrho_1)$ or the non-local Schr\"odinger operator $H$ has no $L^2$-eigenfunctions.
\end{corollary}
\begin{proof}
We have trivially $\dist(x^*, \partial\cK) < \inrad\cK$. Also, $|\lambda| < v$, and $\Psi^{-1}$ is an
increasing function. Hence Theorems \ref{T3.3}-\ref{T3.4} give
$$
\inrad \cK \geq \frac{1}{\sqrt{\Psi^{-1}\left(\frac{v}{\varrho_2}\right)}}\,,
$$
implying the result.
\end{proof}

\begin{remark}
\hspace{100cm}
\begin{trivlist}
\item[\;\,(i)]
We note that, using direct techniques of differential equations, for usual Schr\"odinger operators $H =
-\Delta - v\Ind_{\sB_a}$ in $L^2(\Rd)$, it is well-known that for $d \geq 3$, the smallness of the quantity
$va^2$ implies that no $L^2$-eigenfunctions exist. Using the Birman-Schwinger principle, bounds on $va^\alpha$
can also be derived ruling out $L^2$-eigenfunctions of $H = (-\Delta)^{\alpha/2} - v\Ind_{\sB_a}$ and further
non-local operators \cite{L12}. Although the constants may in general differ, we have the same type of bounds
resulting from Corollary \ref{nogo} above.
\medskip
\item[\;\,(ii)]
We can also use Green functions to find a ``no-go" type consequence, which does not involve \eqref{ET3.3A}.
Suppose that $d\geq 3$ and the transition density  probability function of $\pro X$ decays to $0$ as $t\to\infty$.
Then the ground state $\varphi_1$ of $H = \Psi(-\Delta) - v\Ind_\cK$ has the representation
\begin{equation}\label{ER4.6A}
\varphi_1(x)=\int_{\Rd}(\lambda_1-V(y))\varphi(y)G(x, y)\, \D{y},
\end{equation}
where $G(\cdot, \cdot)$ is the associated Green function. It is known \cite[Th.~3]{Grzywyn-14} that there exists
a constant $C_d$, dependent only on $d$, such that
$$
G(x, y)\leq \frac{C_d}{\abs{x-y}^d\Psi(\abs{x-y}^{-2})}\,.
$$
Let $R=\diam \cK$. Since $x^*\in\cK$, by Theorem~\ref{T4.3}, and $\lambda_1<0$ we see from \eqref{ER4.6A} that
\begin{equation*}
\varphi_1(x^*) \leq C_d\, (v-\abs{\lambda_1})\, \varphi_1(x^*) \int_{\cK} \frac{\D{y}}{\abs{x^*-y}^d\Psi(\abs{x^*-y}^{-2})}\,,
\end{equation*}
which implies
\begin{align*}
1 \leq C_d\, (v-\abs{\lambda_1})\, \int_{\sB_R(x^*)} \frac{\D{y}}{\abs{x^*-y}^d\Psi(\abs{x^*-y}^{-2})}
= C_d\ d\,\omega_d\, v \int_0^{R} \frac{\D{s}}{s\Psi(s^{-2})},
\end{align*}
where $\omega_d$ denotes the volume of the unit ball in $\Rd$. Therefore, if the right hand side is finite
(for example, for $\Psi$ satisfying Assumption~\ref{WLSC}), then there is no ground state whenever
$$
v <  \frac{1}{C_d\ d\ \omega_d\,  \int_0^{R} \frac{\D{s}}{s\Psi(s^{-2})}}.
$$
\end{trivlist}
\end{remark}

\bigskip
Finally we note that our technique in proving Theorem~\ref{T3.3} is also applicable to a more general class of
potentials. Consider equation \eqref{fullev}.
For $V$ convex and increasing we have shown in Theorem~\ref{T4.1} that the maximizer $x^*\in \sU_\lambda = \{x \in \cD:
\, V(x) \leq \lambda\}\cap\cD$. For $\delta>0$ we define the $\delta$-neighborhood of $\sU_\lambda$, i.e.
$$
\sU^\delta_\lambda=\{x\in\Rd\; :\; \dist(x, \sU_\lambda)\leq \delta\}.
$$
The following result provides a sufficient condition for the maximizer to be strictly inside $\sU_\lambda$.

\begin{theorem}
Suppose that $\Psi$ satisfies Assumption \ref{WLSC}. There exist positive constants $\varrho_1$ and $\varrho_2$, dependent only
on $\Psi$, such that if for some $\delta\in (0, \inrad \sU_\lambda)$
\begin{equation*}
\frac{\lambda-\min_{x\in\Rd} V(x)}{\min_{x\in\Rd\setminus\sU^\delta_\lambda} (V(x)-\lambda)} \leq \varrho_1\,,
\end{equation*}
then
$$
\dist(x^*, \partial\sU_\lambda)\geq \frac{1}{\sqrt{\Psi^{-1}\left(\frac{\min_{x\in\Rd\setminus\sB^\delta_\lambda}
(V(x)-\lambda)}{\varrho_2}\right)}}\,.
$$
\end{theorem}

\medskip
\noindent
\textbf{Acknowledgments:}
This research of AB was supported in part by an INSPIRE faculty fellowship and a DST-SERB grant EMR/2016/004810. We thank the
anonymous referee for a careful reading of the manuscript and helpful suggestions.


\end{document}